\documentclass[11pt]{amsart}
\usepackage {fancybox}
\usepackage{graphicx}
\usepackage{ascmac}
\usepackage{tikz-qtree}
\usepackage{forest}
\usepackage{mathrsfs}
\usepackage{amscd,verbatim}
\usepackage{floatflt}
\usepackage{amssymb}
\usepackage{tikz-cd}
\usepackage[all]{xy}
\usepackage[hdivide={20mm,,20mm}, vdivide={20mm,,20mm}]{geometry}
\usepackage{stmaryrd} 
\usepackage[colorlinks,linkcolor=blue,citecolor=blue,urlcolor=red]{hyperref}
\usepackage[normalem]{ulem}

\usepackage{mathtools}

\newcommand{\N}{\mathbb{N}}
\newcommand{\PP}{\mathbb{P}}
\newcommand{\Q}{\mathbb{Q}}
\newcommand{\Z}{\mathbb{Z}}

\newcommand{\sE}{\mathcal{E}}
\newcommand{\sF}{\mathcal{F}}
\newcommand{\sG}{\mathcal{G}}

\newcommand{\sO}{\mathcal{O}}
\newcommand{\sP}{\mathcal{P}}

\newcommand{\sR}{\mathcal{R}}

\newcommand{\ul}[1]{{\underline{#1}}}

\newcommand{\Chow}{\operatorname{\mathbf{Chow}}}
\newcommand{\KMM}{\operatorname{\mathbf{KMM}}}
\newcommand{\KM}{\operatorname{\mathbf{KM}}}

\newcommand{\DM}{\operatorname{\mathbf{DM}}}
\newcommand{\DA}{\operatorname{\mathbf{DA}}}

\newcommand{\logDM}{\operatorname{\mathbf{logDM}}}

\newcommand{\Hom}{\operatorname{Hom}}

\newcommand{\Ker}{\operatorname{Ker}}

\newcommand{\Coker}{\operatorname{Coker}}

\newcommand{\Spec}{\operatorname{Spec}}
\newcommand{\Proj}{\operatorname{\mathbf{Proj}}}
\newcommand{\Sm}{\operatorname{\mathbf{Sm}}}

\newcommand{\eff}{{\operatorname{eff}}}

\newcommand{\Zar}{{\operatorname{Zar}}}

\newcommand{\et}{{\operatorname{\acute{e}t}}}

\newcommand{\ch}{{\operatorname{ch}}}

\newcommand{\CH}{{\operatorname{CH}}}

\renewcommand{\lim}{\operatornamewithlimits{\varprojlim}}

\newcommand{\ol}{\overline}

\renewcommand{\epsilon}{\varepsilon}

\newcommand{\gm}{{\operatorname{gm}}}

\newcommand{\bcube}{{\ol{\square}}}

\newcounter{spec}
{\end{list}}%

\newtheorem{lemma}{Lemma}[section]

\newtheorem{Th}{Theorem}

\newtheorem{RmConj}[Th]{Remark/Conjecture}

\newtheorem{thm}[lemma]{Theorem}

\newtheorem{prop}[lemma]{Proposition}

\newtheorem{cor}[lemma]{Corollary}

\newtheorem{conj}[lemma]{Conjecture}

\theoremstyle{definition}

\newtheorem{defn}[lemma]{Definition}

\theoremstyle{remark}

\newtheorem{example}[lemma]{Example}

\numberwithin{equation}{section}

\title[Derived invariants and motives, Part \rm I\hspace{-.01em}I]{Derived invariants and motives, Part \rm I\hspace{-.01em}I integral derived invariants and some applications}\author{Keiho Matsumoto}

\begin{document}
\maketitle

\begin{abstract}
In this paper we construct new derived invariants with integral coefficients using the theory of motifs, and give several applications. Specifically, we obtain the following results: For complex algebraic surfaces, we prove that certain torsion in the abelianized fundamental group is a derived invariant. We prove that the collection of Hodge-Witt cohomology groups is a derived invariant. In particular, Hodge-Witt reduction and ordinary reduction are preserved by derived equivalence when the characteristic is sufficiently large. Finally, using the techniques of non-commutative algebraic geometry, we prove that Serre's ordinary density conjecture is true for cubic $4$-folds which contain a $\mathbb{P}^2$.
\end{abstract}



\section{Introduction}
Given an invariant $\sF$ of varieties, particularly when $\sF$ comes from a cohomology theory: When is $\sF$ preserved by equivalences of derived categories of coherent sheaves? Even for the Hodge numbers $h^{i,j}$, for example, this question is a complex problem. If $\text{char}(k)=0$, then Hodge numbers are derived invariants when the dimension of varieties is less than $3$ \cite[Corollary C]{derivedHOdge3fold}, and Kontsevich predicted that this continues in higher dimensions. On the other hand, if $\text{char}(k)=p>0$, Antieau and Bragg showed that Hodge numbers are derived invariants for smooth proper surfaces \cite[Thm. 1.3(1)]{AB19}, but Addington and Bragg showed that this does not continues for smooth proper $3$-folds \cite[Thm. 1.1]{Addington17}. For the Abelianization of the Fundamental Group $\pi^{\text{ab}}_1$, Schnell described an example of smooth complex projective $3$-folds $X$ and $Y$, such that $X$ and $Y$ are derived equivalent but $\pi^{\text{ab}}_1$ is not preserved by equivalences of derived categories \cite{fundamentalisnotderived}, this implies that the torsion in the singular cohomology group 
$H^2_{\text{Sing}}(-, \Z)$
with $\Z$-coefficients is not derived invariant. On the other hand, the even and odd degree singular cohomology groups with $\Q$-coefficients are derived invariants. Our goal is to understand the problem by using the motive theory and Grothendieck-Riemann-Roch theory \cite{BorelSerre} and to construct some \textit{integral derived invariants} by using \textit{integral Grothendieck-Riemann-Roch theory} \cite{Papas07}, and \cite{M1}, and we give some applications in some areas.

\subsection{Manin invariants}\label{1.1} Let $k$ be a field and $R$ a commutative ring. We call a graded invariant $\sF=\bigoplus \sF^{j}$ of smooth projective varieties over $k$ with $R$-coefficients which is functorial in Chow groups (details are given in sec. 2) a “\textit{Manin invariant with $R$-coefficients}”. We list some Manin invariants with $R$-coefficients
\begin{itemize}
    \item If $k=\mathbb{C}$, the even and odd degree singular cohomology groups {$H^{\text{even}}_{\text{Sing}}(X, R)$, $H^{\text{odd}}_{\text{Sing}}(X, R)$}.
    \item If $(l,\text{char}(k))=1$ for a prime $l$ and $R=\Z_l$, the Galois representation of the even and odd degree $l$-adic cohomology group %
    \textbf{
    $\bigoplus_{i}H^{2i}_{\acute{e}t}(X_{\ol{k}},\Z_l(i))$,  
    $\bigoplus_{i}H^{2i-1}_{\acute{e}t}(X_{\ol{k}},\Z_l(i)).$
    }%
    \item If $k=R$ admits the resolution of singularities, the collection of Hodge cohomology groups \textbf{$\bigoplus_{j-i=r}H^i_{\Zar}(X,\Omega^j)$} for any $r$.
    \item If $\text{char}(k)=p>0$ and let $R$ be the Witt ring $W$ of $k$, the collection of Hodge-Witt cohomology groups \textbf{$\bigoplus_{j-i=r}H^i(X,W\Omega^j)$} for any $r$.
\end{itemize}

By the same discussion as in \cite{Orlov04}, we prove the following by using the Grothendieck Riemann-Roch theory.
\begin{thm}[{see Sec.\ref{maninInvDerInv}}]
For a $\Q$-algebra $R$, any Manin invariant with $R$-coefficients $\sF=\sF^{\bullet}$ is a derived invariant for smooth projective varieties.
\end{thm}

A conjecture of Orlov \cite[Conj. 1]{Orlov04} implies the following conjecture.

\begin{conj}\label{dream}
For a $\Q$-algebra $R$, each graded piece $\sF^{i}$ of a Manin invariant $\sF^{\bullet}$ with $R$-coefficients is a derived invariant for smooth projective varieties.
\end{conj}

This Theorem/Conjecture could explain why Hodge numbers $h^{i,j}$ of characteristic $0$ should be derived invariants and the singular cohomology group with $\Q$-coefficients should be derived invariants. However, Hodge numbers $h^{i,j}$ of characteristic $p>0$ are not derived invariants, and the torsion subgroup of singular cohomology group with $\Z$-coefficients are not derived invariants. Since the former are graded pieces of Manin invariants with $\Q$-coefficients, the above Theorem/Conjecture would imply that they are derived invariants. On the other hand, since the latter do not have $\Q$-coefficients; we cannot apply the Theorem/Conjecture. We study an integral analogue of this story by using the integral Grothendieck Riemann-Roch theorem \cite{Papas07}, and \cite{M1}. We show that Pappas's result \cite{Papas07} would imply the following.

\begin{thm}[Theorem \ref{ch0Manin}]\label{thm1.3}
If $\text{char}(k)=0$, for a $\Z[\frac{1}{(3d+1)!}]$-algebra $R$, any $d$-Manin invariant with $R$-coefficients $\sF$ is a derived invariant for smooth projective varieties whose dimension are less than $d$.
\end{thm}

We show that our result from \cite{M1} implies the following.

\begin{thm}[Theorem \ref{chpManin}]
If $\text{char}(k)=p>0$, for a $\Z[\frac{1}{(2d+e+1)!}]$-algebra $R$, any $d$-Manin invariant with $R$-coefficients $\sF$ are derived invariants for smooth projective varieties whose dimension are less than $d$ and which can be embedded into $\PP^e_k$. 
\end{thm}
\begin{RmConj}\label{reco}
As with the Conjecture~\ref{dream}, we also predict that if $\text{char}(k)=0$ (resp. $\text{char}(k)=p>0$), for a $\Z[\frac{1}{(3d+1)!}]$-algebra (resp. $\Z[\frac{1}{(2d+e+1)!}]$-algebra) $R$, any graded piece of a Manin invariant with $R$-coefficients $\sF^i$ is a derived invariant for smooth projective varieties whose dimension are less than $d$ (resp. and which can be embedded into $\PP^e_k$).
\end{RmConj}

These theorems imply that the following Manin invariants are preserved by equivalence of derived categories of smooth projective varieties over $k$ whose dimension are less than $d$ and which can be embedded into $\PP^e_k$.
\begin{itemize}
    \item If $k=\mathbb{C}$, the the even and odd degree singular cohomology groups {$H^{\text{even}}_{\text{Sing}}(-, R)$, $H^{\text{odd}}_{\text{Sing}}(-, R)$}
    \item If $\text{char(k)}=0$, for a prime $l>3d +1$, the Galois representations of the even and odd degree $l$-adic cohomology groups
    {$\bigoplus_{i}H^{2i}_{\acute{e}t}(-_{\ol{k}},\Z_l)(i)$, $\bigoplus_{i}H^{2i-1}_{\acute{e}t}(-_{\ol{k}},\Z_l)(i)$}.

    \item If $\text{char(k)}=p$, for a prime $l>2d+e +1$ so that $(p,l)=1$, the Galois representation of the even and odd degree $l$-adic cohomology group {$\bigoplus_{i}H^{2i}_{\acute{e}t}(-_{\ol{k}},\Z_l)(i)$, $\bigoplus_{i}H^{2i-1}_{\acute{e}t}(-_{\ol{k}},\Z_l)(i)$}.
    \item If $\text{char}(k)=p>0$ and $p>2d +e+1$, under the assumption of resolution of singularities, the collection of Hodge cohomology groups {$\bigoplus_{j-i=r}H^i_{\Zar}(X,\Omega^j)$} for any $r$. (Remark. This has been already known for $p\geq \dim X$ \cite{HKRp}).
\end{itemize}

There are certain applications of these results to 
mirror symmetry, algebraic geometry of positive characteristic and number theory. We shall explain these applications.
%

\subsection{\textbf{The application to mirror symmetry}} Batyrev and Kreuzer conjectured that the torsion subgroup in Betti cohomology $\text{Tors}\bigl(H^2_{\text{sing}}(-,\Z)\bigr)$ and $\text{Tors}\bigl(H^3_{\text{sing}}(-,\Z)\bigr)$ are exchanged by homological mirror symmetry \cite{integralmirror}. In \cite{GrossPopescu}, Gross and Popescu gave a counterexample to this conjecture (see \cite{fundamentalisnotderived}). Gross and Popescu constructed a pair of smooth projective varieties $X$ and $Y$ such that $Y$ is to be a (homological) mirror manifold of a (homological) mirror manifold of $X$ \cite{GrossPopescu}, Schnell showed that $X$ and $Y$ are derived equivalent \cite{fundamentalisnotderived}. If the conjecture of Batyrev and Kreuzer holds, then $X$ and $Y$ have the same torsion subgroup of the singular cohomology group $\text{Tor}H^2_{\text{Sing}}(X,\Z)\simeq \text{Tor}H^2_{\text{Sing}}(Y,\Z)$, but this isomorphism does not hold \cite{GrossPopescu}. The problem with the conjecture of Batyrev and Kreuzer is that the torsion subgroup of the singular cohomology group is not a derived invariant. Inspired by Gross and Popescu's counter-example and Theorem~\ref{thm:torsIso} below, we propose the following modified version of the conjecture of Batyrev and Kreuzer.

\begin{conj}
Let $X$ be a complex projective Calabi-Yau $3$-fold, and $M$ be a homological mirror manifold of $X$. Then the $m$-torsion subgroup in Betti cohomology $\bigl(H^2_{\text{sing}}(-,\Z)\bigr)[m]$ and $\bigl(H^3_{\text{sing}}(-,\Z)\bigr)[m]$ are exchanged by $X$ and $M$ for any number $m$ such that $(m,p)=1$ for $p=2,3,5,7$.   
\end{conj}

For a derived equivalence $D^b(X)\simeq D^b(Y)$ of smooth projective varieties, $Y$ is a homological mirror of a mirror manifold of $X$. Thus if the conjecture holds, then $H^2(-,\Z)[m]$ and $H^3(-,\Z)[m]$ should be derived invariants, and Remark/Conjecture~\ref{reco} would imply that these two invariants are derived invariants. We can prove that these two invariants are derived invariants for smooth projective surfaces.

\begin{thm}[Theorem~\ref{GOHANCHAN}] \label{thm:torsIso}
Let $X$ and $Y$ be smooth projective surfaces over $\mathbb{C}$. We assume that there is a $\mathbb{C}$-linear equivalence $F:D^b(X) \simeq D^b(Y)$. For a natural number $m$ such that $(m,p)=1$ for $p=2,3,5,7$ there are isomorphisms of the $m$-torsion parts of cohomology groups
\[
H^i(X,\Z)[m] \simeq H^i(Y,\Z)[m]
\]
for any $i$. In particular, there is an isomorphism of the $m$-torsion parts of the abelianization of the fundamental groups
\[
\pi_1^{\text{ab}}(X)[m]\simeq \pi_1^{\text{ab}}(Y)[m].
\]
\end{thm}

\subsection{The application to algebraic geometry of positive characteristic} For a perfect field $k$, we write $W$ for the Witt ring of $k$ and $K$ for the fraction field of $W$. Antieau-Bragg showed that the collection of Hodge-Witt cohomology groups tensored by $K$ over $W$ of smooth projective varieties:
\[
\bigoplus_{j-i=r}H^i(X,W\Omega^j_X)\otimes_W K
\]
is preserved by equivalences of derived categories for any $r\in\Z$ \cite[Theorem 5.10]{AB19}. Also, they showed that Hodge-Witt reduction of smooth projective varieties whose dimensions are less than $3$ is preserved by equivalences of derived categories \cite[Corollary 5.41]{AB19}. We can prove this continues for higher dimensional and ordinary reduction cases when $p$ is big enough. 

\begin{thm}(Threorem~\ref{4.3})
Let $X$ and $Y$ be smooth projective varieties over a perfect field $k$ with characteristic $p$. Suppose $X,Y\in \Sm\Proj^{\leq d}_{(e)}(k)$ for some $e$ and $d$, and $p>2d+e+1$. We assume that there is a $k$-linear fully faithful triangulated functor $F:D^b(X) \to D^b(Y)$. Then the followings holds.
\begin{enumerate}
    \item For any integer $r$, there is a split injective morphism of $W$-modules 
    \[
    \bigoplus_{j-i=r}H^i(X,W\Omega^j_X) \hookrightarrow \bigoplus_{j-i=r}H^i(Y,W\Omega^j_Y).
    \]
    \item For any integer $r$, if $F$ is an equivalence, then there is an isomorphism of $W$-modules
    \[
    \bigoplus_{j-i=r}H^i(X,W\Omega^j_X) \simeq \bigoplus_{j-i=r}H^i(Y,W\Omega^j_Y).
    \]
    \item If $Y$ is Hodge-Witt, then $X$ is also Hodge-Witt.
    \item If $Y$ is ordinary, then $X$ is also ordinary.
\end{enumerate}
\end{thm}

\subsection{The application to number theory} We consider the following conjecture.
\begin{conj}(Serre conjecture for ordinary density)\label{Serre}
Let $X/K$ be a  smooth projective variety over a  number field $K$. Then there is finite extension $L/K$ such that there exists a positive density of primes $v$ of $L$ for which $X_L$ has a good ordinary reduction at $v$.
\end{conj}
The conjecture is known for elliptic curves by Serre \cite{serreelliptic}, Abelian surfaces by Ogus \cite[Chapter~6]{900}, K3 surfaces by Bogomolov and Zarhin \cite{K3ordinary}, Abelian varieties with complex multiplication and Fermat varieties \cite[Theorem 5.1.10]{Joshi}. We can give a non-commutative approach to this conjecture and prove the following.

\begin{thm}[Theorem \ref{6.4}]
Let $X$ be a cubic $4$-fold which contains $\PP^2$ over a number field $K$. Then $X$ satisfies Conjecture~\ref{Serre}. Moreover, the density of the set of finite primes at which $X_L$ has a good ordinary reduction is one.
\end{thm}

\subsection{Notation and conventions}
We consider the following categories, and ring associated to a field $k$, natural numbers $d,e,r$, and a commutative ring $R$. Where ever possible, we have used notation already existing in the literature.
\begin{itemize}
    \item[] $\Sm\Proj^{\leq d}_{(e)}(k)$: the full subcategory of $\Sm\Proj(S)$ whose objects can be embedded in $\PP^e_k$ and dimension of objects is less than or equal to $d$
   { \item[]{$\KM^{\leq d}_{(e)}(k)$: the smallest full subcategory of $\KM(k)$ which contains the image of the functor $\Sm\Proj^{\leq d}_{(e)}(k) \to \KM(k)$ and is closed under finite coproducts}}
   {\item[]{$\Chow/-\otimes T$: is the orbit category (see~\cite[section 3]{M1}).}} 
   { \item[]{$\Chow(k) \overset{\pi}{\to} \Chow(k)/-\otimes T$: the natural functor from $\Chow(k)$ to $\Chow(k)/-\otimes T$.}}
\end{itemize}

\section{Manin invariants and motives}  We introduce the idea of 
\textit{Manin invariants} and then give a relation between Manin invariants and motive theory. We shall describe the construction of Manin invariants from a functor from the category of effective Chow motives to an additive category. This construction gives us many good examples of Manin invariants.

\subsection{} As in {\S}\ref{1.1}, we fix a base field $k$ and a commutative ring $R$. Given an $R$-linear category $\mathcal{A}$ and functors $F^i:\Sm\Proj(k) \to \mathcal{A}$ for $i\in\Z$. Orlov showed that for a $\Q$-linear functor $G$ from the orbit category $\DM(k,\Q)/T_\Q$ to $\mathcal{A}$ where $T_\Q$ is the Tate twist with $\Q$-coefficients, if smooth projective varieties $X$ and $Y$ are derived equivalent then there is an isomorphism $G(\pi(X)_\Q) \simeq G(\pi(Y)_\Q)$ in $\mathcal{A}$ where $\pi(X)$ is the orbit of $h(X)$ in the orbit category $\DM(k,\Q)/T_\Q$. Motivated by this, we define:
\begin{defn}
Given an $R$-linear category $\mathcal{A}$ which contains any infinite coproducts and convariant functors $F^i:\Sm\Proj(k) \to \mathcal{A}$ for $i\in\Z$, we call $F=\bigoplus_i F^i$ a \textit{Manin invariant} if $F=\bigoplus_i F^i$ satisfies following conditions: Let $X$, $Y$ and $Z$ be smooth projective varieties over $k$.
\begin{itemize}
    \item[(c-1)] For an algebraic cycle $\alpha \in \CH^i(X\times Y)_R$, there is a morphism $F^j(\alpha):F^j(Y) \to F^{j+i-\dim Y}(X)$ in $\mathcal{A}$ for any $j$. In particular, for a morphism of algebraic varieties $f:X \to Y$, the morphism $F^j(Y) \to F^{j}(X)$ comes form the graph morphism $[\Gamma_f(X)]\in\CH^{\dim Y}(X\times Y)$ is equal to the morphism $F^j(f):F^j(Y) \to F^j(X)$ for any $j$.  
    \item[(c-2)] For an algebraic cycle $A= \Sigma \alpha_i \in \CH^*(X\times Y)$ where $\alpha_i$ is a $i$-codimensional cycle in $X\times Y$, we let $F(A)$ denote the morphism $F(Y)=\bigoplus_j F^j(Y) \to F(X)=\bigoplus_l F^l(X)$ in $\mathcal{A}$ whose $(j,l)$-component is given by $F(\alpha_{l+\dim Y -j}): F^j(Y) \to F^l(X)$. If given an algebraic cycle $B$ in $\CH^*(Y\times Z)$ then $F=\bigoplus_i F^i$ satisfies the following equation of morphism in $\mathcal{A}$:
    \[
    F(A) \circ F(B)=F\bigr(r_*(p^*A . q^*B)\bigl)
    \]
    as morphisms from $F(Z)$ to $F(X)$ where $p$ (resp. $r$, $q$) is the projection from $X\times Y \times Z$ to $X\times Y$ (resp. to $Y\times Z$, to $X\times Z$).
\end{itemize}
    We call $F=\bigoplus_i F^i$ a \textit{$d$-Manin invariant} if $F=\bigoplus_i F^i$ satisfies conditions (c-1) and (c-2) for smooth projective varieties $X$, $Y$ and $Z$ whose dimension are less than $d$.
\end{defn}

To relate Manin invariants with motives theory, we now recall the definition of the Manin's category of motives $C_{k,R}$ from \cite{Manin} for a base field $k$ with $R$-coefficients. We let $C_{k,R}$ denote the $R$-linear category whose objects are smooth projective varieties over $k$ and for smooth projective varieties $X$ and $Y$, maps from $X$ to $Y$ are given by the Chow group $\CH^*(X\times Y)_R$ with $R$-coefficients. The composition of two morphisms is defined using intersection product. The identity $id_X$ is given by the class of the diagonal $\Delta_X\in \CH^{\dim X}(X\times X)_R$. Set the full subcategory $C_{k,R}^{\leq d}$ whose objects are varieties whose dimension are less than $d$. Obviously, given a Manin invariant with $R$-coefficients $F=\bigoplus_j F^j$, then there is the functor $\sF:C_{k,R} \to \mathcal{A}$ which sends a smooth projective variety $X$ to $F(X)=\bigoplus_j F^j(X)$ and $A\in \CH^*(X\times Y)$ to $F(A)$ and also given a $d$-Manin invariant with $R$-coefficients $F=\bigoplus_j F^j$ then there is the functor $\sF:C_{k,R}^{\leq d} \to \mathcal{A}$. Conversely, an object $G\in \DM(k,R)$ induces Manin invariants $G_{\text{even}}$ and $G_{\text{odd}}$ for any $m\in \Z$.

\begin{prop}
Given an object $G\in \DM(k,R)$. For a smooth projective variety $X$, we denote $\Hom_{\DM}\bigl(M(X)(j)[2j],G\bigr)$ and $\Hom_{\DM}\bigl(M(X)(j)[2j+1],G\bigr)$ by $G^{j}_{\text{even}}(X)$ and $G_{\text{odd}}^{j}(X)$, respectively. Then $G_{\text{even}}=\bigoplus_{j\in \Z} G^{j}_{\text{even}}$ and $G_{\text{odd}}=\bigoplus_{j\in \Z} G^{j}_{\text{odd}}$ are Manin invariants.
\end{prop}

\begin{proof}
In \cite{V00b}, Voevodsky has proved that there is a fully faithful functor (see \cite[Corollary~4.2.5 and Theorem~4.3.7]{V00b} or \cite[Remark 5.3.21. (iii)]{Kellyldh}):
\begin{equation}\label{voevodsky}
    \Chow^\eff(k)_R \hookrightarrow \DM^{\eff}(k,R).
\end{equation} 
Thus an algebraic cycle $A=\Sigma_{l\in \Z} \alpha_l \in \CH^*(X\times_k Y)$ induces a morphism
\begin{equation}\label{even}
\bigoplus_{j\in\Z} M(X)(j)[2j] \to \bigoplus_{i\in\Z} M(Y)(i)[2i]
\end{equation}
whose $(j,i)$-component is given by $A_{i-j+\dim Y}$. Also $A$ induces a morphism
\begin{equation}\label{odd}
\bigoplus_{j\in\Z} M(X)(j)[2j+1] \to \bigoplus_{i\in\Z} M(Y)(i)[2i+1]
\end{equation}
whose $(j,i)$-component is given by $A_{i-j+\dim Y}$. Morphisms \eqref{even} and \eqref{odd} are functorial in Chow groups, thus functors 
\[
G^i_{\text{even}}: \Sm\Proj(k) \to (R\text{-mod}); \text{  } X \mapsto \Hom_{\DM(k,R)}(M(X)(i)[2i],G)
\]
for $i\in\Z$ satisfies conditions (c-1) and (c-2), and also functors
\[
G^i_{\text{even}}: \Sm\Proj(k) \to (R\text{-mod}); \text{  } X \mapsto \Hom_{\DM(k,R)}(M(X)(i)[2i+1],G)
\]
for $i\in\Z$ satisfies conditions (c-1) and (c-2).
\end{proof}

We shall extend this story to an object $G\in \DM^\eff(k,R)$. By Voevodsky's cancellation theorem \cite{cancel}, we have the fully faithful $R$-linear functor $\DM^\eff(k,R) \hookrightarrow \DM(k,R)$. Define $G^{j}_{\text{even}}(X)$ and $G^{j}_{\text{odd}}(X)$ as 
\[
\Hom_{\DM}(M(X)(j)[2j],G) \text{  and  }\Hom_{\DM}(M(X)(j)[2j+1],G),
\]
respectively, for any smooth projective variety $X$ and $j\in\Z$. Then $G_{\text{even}}=\bigoplus G^{j}_{\text{even}}$ and $G_{\text{odd}}=\bigoplus G^{j}_{\text{odd}}$ are Manin invariants. But we don't wish to do this story, because it is hard to calculate $\Hom_{\DM}(M(X)(j)[2j],G)$ in the case $j<0$. For example, there is an object $\underline{\Omega}^n\in\DM^{\eff}(k,k)$ which represents Hodge cohomology for smooth proper varieties, moreover for $j\geq 0$ we have an isomorphism
\[
\Hom_{\DM}(M(X)(j)[2j],\underline{\Omega}^n[m]) \simeq H^{n-j}(X,\Omega_{X/k}^{n-j}).
\]
On the other hand, for $j< 0$, it is very hard to calculate $\Hom_{\DM}(M(X)(j)[2j],\underline{\Omega}^n[m])$, and actually we don't know it yet:
\[
\Hom_{\DM}(M(X)(j)[2j],\underline{\Omega}^n[m])=~???.
\]
To avoid this, we will assume the object $G$ satisfies a condition on the finiteness of the cohomological dimension.
\begin{defn}
For an object $G\in \DM^\eff(k,R)$, we say $G$ satisfies \textit{finiteness of cohomological dimension} if the following condition holds:
\begin{itemize}
    \item There is an increasing sequence of natural numbers $\{a_d\}_{d\in\Z}$ satisfying the following: for a natural number $d$, given a smooth projective variety $X$ whose dimension is less than $d$, then we have $\Hom_{\DM^\eff}(M(X),G[m])=0$ for any $m> a_d$.
\end{itemize}
We call such a sequence $\{a_n\}$ \textit{cohomological dimension of $G$}.
\end{defn}
\begin{example}
$\{0\}_{n}$ is a cohomological dimension of $\Z\in\DM^\eff(k,\Z)$ since we know
\[
\Hom_{\DM^\eff}(M(X),\Z[m]) \simeq H^{m,0}(X) \simeq \left\{
\begin{array}{ll}
\Z & m=0 \\
0 & m>0
\end{array}
\right.
\]
for any variety $X$. $\{d\}_{d}$ is a cohomological dimension of $\underline{\Omega^m}\in\DM^\eff(k,k)$.
\end{example}

For an $R$-linear category $\mathcal{A}$ containing all products and an $R$-linear triangulated functor $\sG:\DM^\eff(k,R) \to D(\mathcal{A})$ which sends coproducts to products, by Brown representability theorem \cite[Theorem~3.1]{Neeman} there is an object $G\in \DM^\eff(k,R)$ such that following isomorphisms holds for any motive $M\in\DM^\eff(k,R)$ and $m \in \Z$:
\[
\Hom_{\DM^\eff}(M,G[m]) \simeq H^m(\sG(M)).
\]
We shall say that \textit{the object $G\in \DM^\eff$ represents $\mathcal{G}$}. 
\begin{prop}\label{1-1}
We assume $\mathcal{A}$ has all products. Given an object $G\in \DM^\eff(k,R)$ satisfying finiteness of cohomological dimension and which represents an $R$-linear triangulated functor $\DM^\eff(k,R) \to D(\mathcal{A})$, and choose a cohomological dimension $\{a_n\}$ of $G$. Choose natural numbers $d$ and $m>a_{3d}$. We set
\[
G^j_m(X)=\left\{
\begin{array}{ll}
\Hom_{\DM^\eff}(M(X)(j)[2j],G[m])& j>2d \\
~~~0 & j\leq 2d
\end{array}
\right.
\]
for any smooth projective variety $X$ whose dimension is less than $d$. Then 
\[
G_m=\bigoplus G_m^{j}:\Sm\Proj(k) \to \mathcal{A}
\] is a $d$-Manin invariant with $R$-coefficients.
\end{prop}
Before we prove the proposition, we shall give an example. Binda-Park-{\O}stv{\ae}r constructed the Hodge object $\ul{\Omega}^n$ in the triangulated category of log motives $\logDM(k,k)$ for any $n\in\N$ \cite[section~9]{BPO} which represents Hodge cohomology group for log smooth pairs. They also showed that there is an adjunction \cite[(5.2.1)]{BPO}
\[
\xymatrix{
\log\DM^\eff(k,R) \ar@/^15pt/[rr]^-{\omega_{\sharp}}_{} & \bot  & \DM^\eff(k,R) \ar@/^15pt/[ll]^-{R\omega^*}
}.
\]
Under the assumption of resolution of singularities, they also showed that for smooth proper variety $X$ over $k$, there is an isomorphism $R\omega^*M(X)\simeq M(X,\text{triv}_X)$, where $\text{triv}_X$ is the trivial log structure on $X$, and there are adjunctions \cite[proposition~8.2.12]{BPO}
\[
\xymatrix{
\log\DM^\eff(k,R) \ar@/^20pt/[r]^{\omega_{\sharp}}_{\bot} \ar@/_20pt/[r]_{R\omega_{*}}^{\bot} & \DM^\eff(k,R) \ar[l]^{R\omega^*}
}
\]  
In this setting, there is an object $\underline{\Omega_{\log}^n} \in  \logDM^{\eff}(k,k)$ which represents Hodge cohomology group \cite[section 9]{BPO}. Let us denote $\underline{\Omega}^n = R\omega^*\underline{\Omega^n_{\log}}$. For a smooth variety $X$, we choose a good compactification $(\ol{X},X^{\infty})$ of $X$, then there is an isomorphism of $k$-vector spaces:
\[
\Hom_{\DM^\eff(k,k)}(M(X),\underline{\Omega}^n[m]) \simeq H^m(\ol{X},\Omega^n_{\ol{X}/k}(\log |X^\infty|)).
\]
This induces that $\{n\}$ is a cohomological dimension of $\underline{\Omega^n}$. Let us compute the Hodge cohomology group of the Tate twist.  Firstly we recall the residue exact sequence. For an integer $i\geq 1$, we fix a hyperplane $H=X\times\PP^{i-1}\hookrightarrow X\times\PP^{i}$. This regular embedding induces the residue exact sequence:
\[
0 \to \Omega^m_{X\times\PP^i} \overset{a}{\to} \Omega^m_{X \times\PP^i}(\log H) \to \Omega^{m-1}_H \to 0
\]
where $a$ is the natural inclusion. This exact sequence induces the following long exact sequence:
\begin{alignat}{2}\label{BBBBB}
         &\rightarrow H^{n-1}_{\Zar}(H,\Omega^{m-1}_{H}) \rightarrow H^n_{\Zar}(X\times\PP^i,\Omega^m_{X\times\PP^i}) &&\overset{H^n(X\times\PP^i,a)}{\rightarrow} H^n_{\Zar}(X\times\PP^i,\Omega^m_{X\times\PP^i}(\log H))\\
        &\rightarrow H^{n}_{\Zar}(H,\Omega^{m-1}_{H})\rightarrow\cdots. && \nonumber
    \end{alignat}  
\begin{lemma}\label{splitsurj}
The map $H^n(X\times\PP^i,a)$ is a split surjective for all $n$.
\end{lemma} 
\begin{proof}
We take a section $s:X \hookrightarrow X\times\PP^i$ which intersects properly with $H$. Since $s$ intersects properly with $H$, there is the natural map $s^*:H^n_{\Zar}(X\times\PP^i,\Omega^m_{X\times\PP^i}(\log H)) \to H^n_{\Zar}(X,\Omega^m_{X})$ induced by $s$, and it is an isomorphism \cite[corollary~9.2.2]{BPO}. We call this isomorphism $\bcube$-invariance. 
There is the following sequence of varieties
\[
\xymatrix{X  \ar@/^12pt/[rr]^{id} \ar[r]_-s &  X\times\PP^i \ar[r]_{f} & X
}
\]
where $f$ is the projection. This induces the following sequence of $k$-vector spaces
\begin{equation}\label{muchi}
\xymatrix{ H^n(X,\Omega^m_X) &  H^n(X\times\PP^i,\Omega^m_{X\times\PP^i}) \ar[l]_-{H^n(s,\Omega^m)} & H^n(X,\Omega^m_X) \ar@/^12pt/[ll]^{id} \ar[l]_-{H^n(f,\Omega^m)}
}
\end{equation}
where $H^n(,\Omega^m)$ (resp. $H^n(f,\Omega^m)$ ) is the map induced by the natural map $\Omega^m_{X\times\PP^i} \to s_* \Omega^m_X$ (resp. $\Omega^m_{X} \to f_* \Omega^m_{X\times\PP^i}$).
Since the map $H^n(s,\Omega^m)$ is the composition of $H^n(X\times\PP^i,a)$, and $s^*$, by the $\bcube$-invariance of $s^*$ we have the following sequence:
\begin{equation}\label{kato}
\xymatrix{ H^n(X,\Omega^m_X) & H^n_{\Zar}(X\times\PP^i,\Omega^m_{X\times\PP^i}(\log H)) \ar[l]_-{\simeq}^-{s^*} & H^n(X\times\PP^i,\Omega^m_{X\times\PP^i}) \ar[l]_-{H^n(X\times\PP^i,a)} \ar@/_25pt/[ll]^{H^n(s,\Omega^m)}& H^n(X,\Omega^m_X) \ar@/^25pt/[lll]^{id} \ar[l]_-{H^n(f,\Omega^m)}.
}
\end{equation}
Thanks to this sequence, we obtain the claim.
\end{proof}
We keep the notation. By the diagram (\ref{muchi}) we obtain the equality $\ker(H^n(s,\Omega^m))=\Coker(H^n(f,\Omega^m))$, and by the diagram (\ref{kato}) we have the equality \begin{equation*}
\Ker(H^n(X\times\PP^i,a))\simeq\Ker(H^n(s,\Omega^m))=\Coker(H^n(f,\Omega^m)).
\end{equation*}
Besides by lemma~\ref{splitsurj} and the long exact sequence~(\ref{BBBBB}) we obtain the following:
\begin{equation}\label{KERKER}
H^{n-1}_{\Zar}(H,\Omega^{m-1}_H) \simeq \Ker(H^n(X\times\PP^i,a))\simeq\Coker(H^n(f,\Omega^m)).
\end{equation}
In the case $i=1$, the hyper plain $H=X$, the projection $f:X\times\PP^1 \to X$ induces an isomorphism in $\DM^\eff(k,k)$:
\[
M(X\times\PP^1 )_k \simeq M(X)_k\oplus M(X)_k(1)[2]
\]
where the projection to the first factor $M(X\times\PP^1)_k \to M(X)_k$ is the natural map $M(f)$. This decomposition induces the following isomorphism
\begin{equation}\label{AAAAA}
\Coker({\DM^\eff(k,k)}(M(f),\ul{\Omega^m}[n])) \simeq {\DM^\eff(k,k)}(M(X)(1)[2],\ul{\Omega^m}[n]).\end{equation}
We know that the object $\ul{\Omega^m}$ represents the Hodge cohomology, thus we have the following:
\[
\Coker({\DM^\eff(k,k)}(M(f),\ul{\Omega^m}[n])) \simeq \Coker(H^n(f,\Omega^m)),
\]
thanks to this isomorphism, (\ref{KERKER}) and (\ref{AAAAA}), we obtain the following
\begin{equation}\label{TATESHIFT}
H^{n-1}_{\Zar}(X,\Omega^{m-1}_X) \simeq {\DM^\eff(k,k)}(M(X)(1)[2],\ul{\Omega^m}[n]).
\end{equation}
Now we start to compute the Hodge realization of the Tate twist $M(X)_k(i)[2i]$. At first, we prove the following:
\begin{prop}\label{ISHIKAWA} For a proper smooth variety $X$, there is an isomorphism of $k$-vector spaces:
\begin{eqnarray*}
{\DM^\eff(k,k)}(M(X)_k(i)[2i],\ul{\Omega^m}[n])\simeq {\DM^\eff(k,k)}(M(X)_k(i-1)[2i-2],\ul{\Omega^{m-1}}[n-1]).
\end{eqnarray*}
\end{prop}
\begin{proof}
We will use a projection formula and $(\PP^n,\PP^{n-1})$-invariance of the Hodge cohomology. Consider projections $f_i: X\times \PP^i \to X$ and $f_{i-1}:X\times \PP^{i-1} \to X$ and fix a hyper plain $j:X\times \PP^{i-1} \to X\times\PP^i$, take another hyper plain $H\simeq X\times\PP^{i-1}$ of $X\times \PP^i$ intersects properly with $j$.
\[
\xymatrix{
j^*H \ar[rr]^{j'} \ar[d] && H \ar[d] \\
X\times \PP^{i-1} \ar[rr]^j \ar[rd]_-{f_{i-1}} && X\times\PP^i  \ar[ld]^-{f_{i}}\\
&X  &
}
\]
Thanks to the projection formula, the cone of $M(j):M(X\times\PP^{i-1})_k \hookrightarrow M(X \times \PP^i)_k$ is equal to $M(X)_k(i)[2i]$ and also the cone of $M(j'):M(j^*H)_k \hookrightarrow M(H)_k$ is equal to $M(X)_k(i-1)[2i-2]$. Let us study a morphism between the residue sequences induced by $j$:
\[
\xymatrix{
0 \ar[r] & \Omega^m_{X\times \PP^i} \ar[r] \ar[d]& \Omega^m_{X\times \PP^i}(\log H) \ar[r] \ar[d]& \Omega^{m-1}_{H} \ar[r] \ar[d]& 0\\
0 \ar[r]& j_*\Omega^m_{X\times\PP^{i-1}} \ar[r]& j_*\Omega^m_{X\times\PP^{i-1}}(\log j^*H) \ar[r]& j_*\Omega^{m-1}_{j^*H} \ar[r]& 0
}
\]
Now we know $f_{i-1}$ and $f_{i}$ induce an isomorphism $H^n(X\times\PP^{i-1},\Omega^m_{X\times\PP^{i-1}}(\log j^*H)) \simeq H^n(X,\Omega^m_X)$ and an isomorphism $H^n(X \times \PP^i,\Omega^m_{X \times \PP^i}(\log H)) \simeq H^n(X,\Omega^m_X)$. By Lemma~\ref{splitsurj} there is a morphism between split exact sequences:
\begin{equation}\label{ninenine}
\xymatrix{
0 \ar[r] & H^{n-1}(H,\Omega^{m-1}_{H}) \ar[r] \ar[d]^{H^{n-1}(j',\Omega^{m-1})}& H^n(X\times \PP^i,\Omega^m_{X\times \PP^i}) \ar[r] \ar[d]^{H^n(j,\Omega^m)}& H^n(X\times \PP^i,\Omega^m_{X\times \PP^i}(\log H)) \ar[r] \ar[d]^{\simeq}& 0\\
0 \ar[r]& H^{n-1}(j^*H,\Omega^{m-1}_{j^*H}) \ar[r]& H^n(X\times\PP^{i-1},\Omega^m_{X\times\PP^{i-1}}) \ar[r]& H^n(X\times\PP^{i-1},\Omega^m_{X\times\PP^{i-1}}(\log j^*H)) \ar[r]& 0
}
\end{equation}
Since morphisms $M(j'):M(j^*H) \to M(H)$ and $M(j):M(X\times\PP^{i-1}) \to M(X \times \PP^{i})$ are split injective, we know that the map $H^{n-1}(j',\Omega^{m-1})\simeq {\DM^\eff(k,k)}(M(j'),\ul{\Omega^{m-1}}[n-1])$ and the map $H^{n}(j,\Omega^{m})\simeq {\DM^\eff(k,k)}(M(j),\ul{\Omega^m}[n])$ are split surjective. Now we apply the nine lemma to the diagram (\ref{ninenine}) we obtain that
\[
\Ker(H^{n-1}(j',\Omega^{m-1}))\simeq \Ker(H^{n}(j,\Omega^{m})),
\]
thanks to this isomorphism we have the following
\begin{eqnarray*}
 {\DM^\eff(k,k)}(M(X)_k(i-1)[2i-2],\ul{\Omega^{m-1}}[n-1]) &\simeq & \Ker({\DM^\eff(k,k)}(M(j'),\ul{\Omega^{m-1}}[n-1])) \\   
&\simeq& \Ker(H^{n-1}(j',\Omega^{m-1})) \\
&\simeq& \Ker(H^{n}(j,\Omega^{m})) \\
&\simeq &  \Ker({\DM^\eff(k,k)}(M(j),\ul{\Omega^{m}}[n])) \\
&\simeq &  {\DM^\eff(k,k)}(M(X)_k(i)[2i],\ul{\Omega^m}[n]).
\end{eqnarray*}
This is what we want. 
\end{proof}
As a corollary of this proposition, we obtain the following:
\begin{cor}\label{STisom2} Take an integer $m \geq 1$. For an integer $m\geq i\geq 0$ and any $n\in \Z$ and any smooth proper variety $X$, there is an isomorphism of $k$-vector space
\begin{equation}\label{STSTST}
{\DM^\eff(k,k)}(M(X)_k(i)[2i],\ul{\Omega^m}[n])\overset{\phi^n_{X,i}}{\simeq} H^{n-i}_{\Zar}(X,\Omega^{m-i}_{X}),
\end{equation}
and for an integer $i>m$ and an integer $n\in\Z$ there is a vanishing
\begin{equation}\label{vanishing}
{\DM^\eff(k,k)}(M(X)_k(i)[2i],\ul{\Omega^m}[n]) =0
\end{equation}
\end{cor}
\begin{proof}
In the case $m\geq i\geq 0$, thanks to proposition~\ref{ISHIKAWA} we have the following:
\begin{eqnarray*}
 {\DM^\eff(k,k)}(M(X)_k(i)[2i],\ul{\Omega^m}[n])  &\simeq & {\DM^\eff(k,k)}(M(X)_k(i-1)[2i-2],\ul{\Omega^{m-1}}[n-1])\\  
&\simeq& {\DM^\eff(k,k)}(M(X)_k(i-2)[2i-4],\ul{\Omega^{m-2}}[n-2]) \\
&\simeq& \cdots \\
&\simeq &  {\DM^\eff(k,k)}(M(X)_k,\ul{\Omega^{m-i}}[n-i]) \\
&\simeq & H^{n-i}_{\Zar}(X,\Omega^{m-i}_X).
\end{eqnarray*}
Thus the statement (\ref{STSTST}) holds. In the case $i>m$, similarly thanks to proposition~\ref{ISHIKAWA} we have the following:
\begin{eqnarray*}
 {\DM^\eff(k,k)}(M(X)_k(i)[2i],\ul{\Omega^m}[n])  &\simeq & {\DM^\eff(k,k)}(M(X)_k(i-1)[2i-2],\ul{\Omega^{m-1}}[n-1])\\  
&\simeq& {\DM^\eff(k,k)}(M(X)_k(i-2)[2i-4],\ul{\Omega^{m-2}}[n-2]) \\
&\simeq& \cdots \\
&\simeq &  {\DM^\eff(k,k)}(M(X)_k(i-m)[2i-2m],\ul{\Omega^{0}}[n-m]).
\end{eqnarray*}
Consider the projection $f:X\times\PP^{i-m} \to X$. The morphism $f$ induces the isomorphism $H^l(X,\sO_X) \simeq H^l(X\times\PP^{i-m},\sO_{X\times\PP^{i-m}})$ for all $l$, hence the map $\DM^\eff(k,k)(M(f),\ul{\Omega^0}[l])$ is an isomorphism for all $l$. On the other hand, thanks to the projection formula, we know \[\Coker\bigl(\DM^\eff(k,k)(M(f),\ul{\Omega^0}[l])\bigr)\simeq \DM^\eff(k,k)(M(X)(1)[2]\oplus \cdots \oplus M(X)(i-m)[2i-2m],\ul{\Omega^0}[l]),
\]
thus we obtain the vanishing:
\[
\DM^\eff(k,k)(M(X)(1)[2]\oplus \cdots \oplus M(X)(i-m)[2i-2m],\ul{\Omega^0}[l])=0
\]
for all $l$. This implies the vanishing ${\DM^\eff(k,k)}(M(X)_k(i-m)[2i-2m],\ul{\Omega^{0}}[n-m])=0$. We can finish the proof. 
\end{proof}

Let us apply Proposition~\ref{1-1} for $G=\underline{\Omega^n}$. For $m>3d$ and a smooth projective variety $X$ whose dimension is less than $d$, by Corollary~\ref{STisom2} we have
\[
G^j_m(X)=\left\{
\begin{array}{ll}
\Hom_{\DM^\eff}(M(X)(j)[2j],\underline{\Omega^n}[m]) \simeq H^{m-j}(X,\Omega^{n-j}_{X/k})& j>2d \\
~~~0 & j\leq 2d
\end{array}
\right.
\]
Thus we know $G_m=\bigoplus_j G^j_m= \bigoplus_{j>d} H^{m-j}(X,\Omega^{n-j}_{X/k})\oplus \bigoplus_{j\leq d} 0 \simeq  \bigoplus_{j\in \Z} H^{m-j}(X,\Omega^{n-j}_{X/k})$ is a $d$-Manin invariant with $k$-coefficients, where we use $H^{m-j}(X,\Omega^{n-j}_{X/k})=0$ for $j\leq d$ since $m-j>d$.

\begin{proof}[Proof of Proposition \ref{1-1}]
Given smooth projective varieties $X$, $Y$ and $Z$ over $k$. For algebraic cycles $A=\Sigma_i \alpha_i \in \CH^*(X\times_k Y)$ and $B=\Sigma_i \beta_i \in \CH^*(Y\times_k Z)$, there are morphisms in $\DM(k,R)$:
\[
\bigoplus_{i\in \Z} M(X)(i)[2i] \overset{A}{\to} \bigoplus_{l\in \Z} M(Y)(l)[2l] \quad\text{  and  } \quad\bigoplus_{l\in \Z} M(Y)(l)[2l] \overset{B}{\to} \bigoplus_{e\in \Z} M(Z)(e)[2e].
\]
Let us restrict these morphisms. Voevodsky proved the following isomorphism \cite{V00b}
\[
\Hom_{\DM(k,R)}(M(X)(i)[2i],M(Y)(l)[2l]) \simeq \CH^{\dim Y + l-i}(X\times_k Y)_R.
\]
Since we assume $\max\{\dim X, \dim Y \}\leq d$, if $|l-i|>d$ then $\CH^{\dim Y + l-i}(X\times_k Y)_R=0$, thus we have 
\[
\Hom_{\DM(k,R)}(M(X)(i)[2i],M(Y)(l)[2l])  =0 \quad \text{ when }\quad |l-i|>d,
\]
and we have a commutative diagram
\[\xymatrix{
\displaystyle\bigoplus_{i\in\Z}M(X)(i)[2i] \ar@{->>}[d] \ar[r]^{A}  & \displaystyle\bigoplus_{l\in\Z}M(Y)(l)[2l] \ar@{->>}[d] \ar[r]^{B} & \displaystyle\bigoplus_{e\in\Z}M(Z)(e)[2e] \ar@{->>}[d] \\
\displaystyle\bigoplus_{i\geq 2d}M(X)(i)[2i] \ar[r] \ar[r]^{\tilde{A}}& \displaystyle\bigoplus_{l\geq d}M(Y)(l)[2l] \ar[r]^{\tilde{B}}& \displaystyle\bigoplus_{e\geq 0}M(Z)(e)[2e]  }\]
For any smooth projective variety $Z$ whose dimension is less than $d$, since we know $\dim \PP^d\times_k W \leq 3d$ and $m>a_{3d}$ thus we have \[\Hom_{\DM^{\eff}}(M(\PP^d \times_k W),G[m])=0.\]  
For $2d \geq l \geq 0$, $M(W)(l)[2l]$ is a direct summand of $\PP^{2d}\times_k W$, thus we have an equality \[
\Hom_{\DM^{\eff}}(M(W)(l)[2l],G[m])=0 \quad\text{ for any } \quad 0 \leq l \leq 2d.
\]

 We shall apply the functor $\Hom_{\DM}(-,G[m])$ to the above diagram.\footnotesize \[\xymatrix@C=6pt{
\displaystyle\bigoplus_{i\in\Z} \Hom_{\DM^{\eff}}(M(X)(i)[2i],G[m])     & \displaystyle\bigoplus_{l\in\Z}\Hom_{\DM^{\eff}}(M(Y)(l)[2l],G[m])  \ar[l] & \displaystyle\bigoplus_{e\in\Z}\Hom_{\DM^{\eff}}(M(Z)(e)[2e],G[m])  \ar[l]  \\
\displaystyle\bigoplus_{i\geq 2d}\Hom_{\DM^{\eff}}(M(X)(i)[2i],G[m]) \ar@{^{(}-_>}[u]^{\text{split}} & \displaystyle\bigoplus_{l\geq 2d}\Hom_{\DM^{\eff}}(M(Y)(l)[2l],G[m])  \ar@{^{(}-_>}[u]^{\text{split}} \ar[l] & \displaystyle\bigoplus_{e\geq 2d}\Hom_{\DM^{\eff}}(M(Z)(e)[2e],G[m]) \ar[l] \ar@{^{(}-_>}[u]^{\text{split}} }\]\normalsize
where vertical maps are split injective. Since the vertical maps preserve the graded structure, we obtain the claim.
\end{proof}
We shall extend this story to an $R$-linear functor $\DM_{\text{gm}}^{\eff}(k,R) \to D(\mathcal{A})$. \begin{defn}
For an $R$-linear functor $\Gamma:\DM_{\text{gm}}^\eff(k,R) \to D(\mathcal{A})$, we say $\Gamma$ satisfies \textit{finiteness of cohomological dimension} if the following condition holds:
\begin{itemize}
    \item There is an increasing sequence of natural numbers $\{a_d\}_{d\in\Z}$ satisfying the following: for a natural number $d$, given a smooth projective variety $X$ whose dimension is less than $d$, then $H^m(\Gamma(M(X)))=0$ for any $m> a_d$.
\end{itemize}
We call such a sequence $\{a_n\}$ \textit{cohomological dimension of $\Gamma$}.
\end{defn} 
By the same discussion, we can prove the following.
\begin{prop}\label{2}
We assume $\mathcal{A}$ has all coproducts. Given an $R$-linear triangulated functor $\Gamma:\DM_{\text{gm}}^\eff(k,R) \to D(\mathcal{A})$ which satisfies finiteness of cohomological dimension, and choose a cohomological dimension $\{a_n\}$ of $\Gamma$. Choose a natural numbers $d$ and $m>a_{3d}$. We set
\[
\Gamma^j_m(X)=\left\{
\begin{array}{ll}
H^m\bigl(\Gamma(M(X)(j)[2j])\bigr)& j>2d \\
~~~0 & j\leq 2d
\end{array}
\right.
\]
for any smooth projective variety $X$ whose dimension is less than $d$. Then 
\[
\Gamma_m=\bigoplus \Gamma_m^{j}:\Sm\Proj(k) \to \mathcal{A}
\] is a $d$-Manin invariant with $R$-coefficients.
\end{prop}

\subsection{} \label{maninInvDerInv} In this section, we shall show Manin invariants relates to non-commutative algebraic geometry. Let us recall the category $\KM(k)$ of Gillet-Soul\'e’s $K$-motives from \cite[Def.5.1, 5.4, 5.6]{KMOTIVE} for a field $k$ and a commutative ring $R$. The category $\KM(k)_R$ is the idempotent completion of the category whose objects are the regular projective $k$-varieties over $k$ and whose morphisms, for regular projective $k$-varieties $X, Y$, are given by
\[
\KM(k)_R(X,Y) := K_0(X\times Y)_R.
\]
Composition is defined as follows:
\begin{eqnarray*}
\KM(k)_R(X,Y) \times \KM(k)_R(Y,Z) & \to& \KM(k)_R(X\times Z) \\
{[\sF]} \times {[\sG]} &\mapsto& {{p_{13}}_*[p_{12}^*\sF\otimes^{\mathbb{L}}p_{23}^*\sG]}.
\end{eqnarray*}
The identity of $X$ is given by the object in $\KM(k)_R(X,X)$ corresponding to $[{\Delta_X}_*\sO_X]\in K_0(X\times X)_R$ where $\Delta_X$ is the diagonal morphism of $X$. By Grothendieck Riemann-Roch theory \cite{BorelSerre}, if $R$ is a $\Q$-algebra there is an $R$-linear functor (see \cite{Tab14} for the details)
\[
\omega_{k,R}:\KM(k)_R \to C_{k,R}.
\]
For a derived equivalence of smooth projective varieties $F:D^b(X) \overset{\simeq}{\to} D^b(Y)$ over $k$, there is the object $P\in D^b(X\times Y)$ which represents $F$ and induces an isomorphism $X \overset{P}{\simeq } Y$ in $\KM(k)_R$ (see \cite{Orlov03} for the details). Given a Manin invariant $F=\bigoplus F^\bullet:\Sm\Proj(k) \to \mathcal{A}$ with $R$-coefficients, then there is a functor $\sF:C_{k,R} \to \mathcal{A}$. Thus if $R$ is a $\Q$-algebra then we have a sequence of functors
\[
\KM(k)_R \overset{\omega_{k,R}}{\to} C_{k,R} \overset{\sF}{\to} \mathcal{A}.
\]
Furthermore, we obtain the following result.
\begin{thm}
For a $\Q$-algebra $R$, any Manin invariant with $R$-coefficients $\sF=\sF^{\bullet}$ is a derived invariant for smooth projective varieties.
\end{thm}

\begin{proof}
    For a derived equivalence of smooth projective varieties $D^b(X)\simeq D^b(Y)$, by \cite[Lemma 3.2]{M1} we have an isomorphism $X_R \simeq Y_R$ in $\KM(k)_R$. Thus we obtain an isomorphism $F(X)=\sF\circ \omega_{k,R}(X) \simeq \sF\circ \omega_{k,R}(Y) =F(Y)$ in $\mathcal{A}$.
\end{proof}
To continue this story when $R$ is not a $\Q$-algebra we have studied integral Grothendieck Riemann-Roch theory \cite[Theorem 1.1]{M1}. Consider the full subcategory $\Sm\Proj^{\leq d}(k)$ of $\Sm\Proj(S)$ such that dimension of objects are less than or equal to $d$, and the full subcategory $\Sm\Proj^{\leq d}_{(e)}(k)$ of $\Sm\Proj^{\leq d}(S)$ whose objects can be embedded in $\PP^e_k$. Now we define $\KM^{\leq d}(k)$ (resp. $\KM^{\leq d}_{(e)}(k)$) as the smallest full subcategory of $\KM(k)$ which contains the image of the functor $\Sm\Proj^{\leq d}(k) \to \KM(k)$ (resp. $\Sm\Proj^{\leq d}_{(e)}(k) \to \KM(k)$) and is closed under finite coproducts. We recall the comparison between non-commutative motives and Chow motives with integral coefficients which is proved in \cite[Theorem 3.1]{M1}. For natural numbers $d$ and $e$, if $\ch(k)=0$ let $R$ be a $\Z[\frac{1}{(3d+1)!}]$-algebra (resp. if $\ch(k)=p$ let $R$ be a $\Z[\frac{1}{(2d+e+1)!}]$-algebra), we have constructed an $R$-linear functor (see \cite[Corollary A.3]{M1} (resp. \cite[Theorem 3.1]{M1}))
 \begin{equation}\label{intkont}
     \Phi_{R}: \KM^{\leq d}(k)_R \to \Chow(k)_R/-\otimes T_R ~\text{  (resp. }  \Phi_{R}: \KM^{\leq d}_{(e)}(k)_R \to \Chow(k)_R/-\otimes T_R).
 \end{equation}
 Via this functor \eqref{intkont}, the fully faithful functor $\theta: C_{k,R} \to  \Chow(k)_R/-\otimes T_R$ induces the following $R$-linear functor
\begin{eqnarray}\label{ONIONI}
\omega_{k,R}:\KM^{\leq d}(k)_R \to {C_{k,R}^{\leq d}} ~\text{  (resp. } \omega_{k,R}:\KM_{(e)}^{\leq d}(k)_R &\to& {C_{k,R}^{\leq d}}) 
\end{eqnarray}
For a derived equivalent smooth projective varieties $F:D^b(X) \overset{\simeq}{\to} D^b(Y)$ over $k$, there is an isomorphism $X \overset{P}{\simeq } Y$ in $\KM(k)_R$. Given a $d$-Manin invariant $F=\bigoplus F^\bullet:\Sm\Proj(k) \to \mathcal{A}$ with $R$-coefficients, then there is a functor $\sF:C_{k,R}^{\leq d} \to \mathcal{A}$. Thus we have a sequence of functors
\[
\KM^{\leq d}(k)_R \overset{\omega_{k,R}}{\to} C_{k,R}^{\leq d} \overset{\sF}{\to} \mathcal{A} ~\text{  (resp. } \KM^{\leq d}_{(e)}(k)_R \overset{\omega_{k,R}}{\to} C_{k,R}^{\leq d} \overset{\sF}{\to} \mathcal{A}).
\]
Thus we obtain the following result.

\begin{thm}\label{ch0Manin}
If $\text{char}(k)=0$, for a $\Z[\frac{1}{(3d+1)!}]$-algebra $R$, any $d$-Manin invariant with $R$-coefficient $\sF$ is a derived invariant for smooth projective varieties whose dimension are less than $d$.
\end{thm}

\begin{thm}\label{chpManin}
If $\text{char}(k)=p>0$, for a $\Z[\frac{1}{(2d+e+1)!}]$-algebra $R$, any $d$-Manin invariant with $R$-coefficient $\sF$ is a derived invariant for smooth projective varieties whose dimension are less than $d$ and which can be embedded into $\PP^e_k$. 
\end{thm}

\section{The application to mirror symmetry}
This section assumes the base field $k$ is of characteristic $0$. For an embedding of fields $\sigma:k \hookrightarrow \mathbb{C}$, there is the Betti realization with $R$-coefficients of Chow motives:
\begin{eqnarray*}\label{BABY1}
    \Gamma^{R}_{\text{sing}}:\DM^{\eff}(k,R) \overset{\sigma^*}{\to} \DM^{\eff}(\mathbb{C},R) \overset{\text{Ro}_{\text{tr}}}{\simeq} \DA^{\eff,\et}(\mathbb{C},R) \overset{\text{Bti}(\mathbb{C})}{\to} D(R\text{-mod}),
\end{eqnarray*}
where $\sigma^*$ is the natural base change functor along $\sigma$, and $\text{Bti}(\mathbb{C})$ is the Betti realization of motives constructed in \cite[Définition 1.7]{Ayoubbetti} and the equivalence $\text{Ro}_{\text{tr}}$ is proved by Ayoub \cite[Theorem 4.4]{AyoubICM}. For a smooth projective variety $X$ over $k$, the functor $\Gamma_{\text{sing}}^R$ induces an isomorphism of $R$-modules
\[
H^i\bigl(\Gamma^{R}_{\text{sing}}(M(X)_{R}) \bigr) \simeq H^i_{\text{sing}}(X_{\sigma}^{\text{an}},R),
\]
and for any positive integer $j\in\Z_{\geq 0}$ there is an isomorphism of $R$-modules
\begin{equation}\label{booboo}
H^i\bigl(\Gamma^{R}_{\text{sing}}(M(X)_{R}(j)[2j]) \bigr) \simeq H^{i-2j}_{\text{sing}}(X_{\sigma}^{\text{an}},R).
\end{equation}
Let us denote $R_{\text{Bti}}\in \DM^\eff(k,R)$ the object which represents $\Gamma^R_{\text{sing}}$. We know $\{2n\}_{n}$ is a cohomological dimension of $R_{\text{Bti}}$. Choose a natural number $d$. For a natural number $m>6d$, we set 
\[
R_{\text{Bti},m}^j(X) = \left\{
\begin{array}{ll}
\Hom_{\DM^\eff}(M(X)(j)[2j],R_{\text{Bti}}[m])\simeq H^{m-2j}_{\text{sing}}(X_{\sigma}^{\text{an}},R) & j>2d \\
~~~0 & j\leq 2d
\end{array}
\right.
\]
If $\dim X \leq d$, for $m >6d$, then we have an isomorphism
\[
R_{\text{Bti},m}(X) = \bigoplus_{j>2d} R_{\text{Bti},m}^j(X) \simeq \bigoplus_{j>2d} H^{m-2j}_{\text{sing}}(X_{\sigma}^{\text{an}},R) = \left\{
\begin{array}{ll}
\bigoplus_{i:\text{even}} H^{i}_{\text{sing}}(X_{\sigma}^{\text{an}},R)  & \text{if }m \text{:even} \\
\bigoplus_{i:\text{odd}} H^{i}_{\text{sing}}(X_{\sigma}^{\text{an}},R) & \text{if }m \text{:odd}
\end{array}
\right.
\]
since $H^j(X_\sigma^\text{an},R)=0$ for any $j>2d$, and by Proposition~\ref{1-1} we know $R_{\text{Bti},m}$ is a $d$-Manin invariant. If $R$ is a $\Z[\frac{1}{(3d+1)!}]$-algebra then by the sequence of functors \eqref{ONIONI} there is a functor 
\begin{eqnarray}\label{redrock}
R_{\text{Bti},m}:\KM^{\leq d}(k)_R \to (R\text{-mod})
\end{eqnarray}
\begin{thm}\label{GOHANCHAN}
We fix an embedding of fields $\sigma:k \hookrightarrow \mathbb{C}$. Let $X$ and $Y$ be smooth projective varieties over $k$. We denote $d=\dim Y$. We assume that there is a $k$-linear fully faithful triangulated functor $F:D^b(X) \to D^b(Y)$. Then the following holds.
\begin{itemize}
 \item[(1)] There are split injective morphism of $\Z[\frac{1}{(3d+1)!}]$-module 
    \[
    \Phi^{even}_{sing}(F_{\sigma}):H_{\text{sing}}^{\text{even}}(X_{\sigma},\Z[\frac{1}{(3d+1)!}]) \hookrightarrow H_{\text{sing}}^{\text{even}}(Y_{\sigma},\Z[\frac{1}{(3d+1)!}]),
    \]
    \[
    \Phi^{odd}_{sing}(F_{\sigma}):H_{\text{sing}}^{\text{odd}}(X_{\sigma},\Z[\frac{1}{(3d+1)!}]) \hookrightarrow H_{\text{sing}}^{\text{odd}}(Y_{\sigma},\Z[\frac{1}{(3d+1)!}]).
    \]
    \item[(2)] If $F$ is an equivalence then $\Phi^{even}_{sing}(F_{\sigma})$ and $\Phi^{odd}_{sing}(F_{\sigma})$ are isomorphisms.
    \end{itemize}
\end{thm}
\begin{proof}
We denote $R=\Z[\frac{1}{(3d+1)!}]$. By the assumption, we know $\dim X\leq d$. Thus $X_R,Y_R \in \KM^{\leq d}(k)_R$ for some $e\in \N$. For a fully faithful functor (resp. equivalent functor) $F:D^b(X)\to D^b(Y)$, there is a split injective map (resp. isomorphism) 
\[
X_R \hookrightarrow Y_R \text{  (resp. }X_R\simeq Y_R )
\]
in $\KM^{\leq d}(k)_R$ (see \cite[Lemma 3.2]{M1}). The claim follows from the functor \eqref{redrock}.
\end{proof}
\begin{thm}
Let $X$ and $Y$ be smooth projective surfaces over $\mathbb{C}$. We assume that there is a $\mathbb{C}$-linear equivalence $F:D^b(X) \simeq D^b(Y)$. For a natural number $m$ such that $(m,p)=1$ for $p=2,3,5,7$ then there are isomorphisms of the $m$-torsion parts of cohomology groups
\[
H^i(X,\Z)[m] \simeq H^i(Y,\Z)[m]
\]
and equalities
\[
H^i(X,\Z) \simeq H^i(Y,\Z)[m]
\]
for any $i$. In particular, there is an isomorphism of the $m$-torsion parts of the abelianization of the fundamental groups
\[
\pi_1^{\text{ab}}(X)[m]\simeq \pi_1^{\text{ab}}(Y)[m].
\]
\end{thm}
\begin{proof}
For simplicity  of  notation we write $R$ instead of $\Z[\frac{1}{7!}]$. For a natural number $m$ such that $(m,p)=1$ for $p=2,3,5,7$, by Theorem~\ref{GOHANCHAN} there is an isomorphism of $m$-torsion groups
\[
H_{\text{sing}}^{+}(X_{\sigma},R)[m] \simeq H_{\text{sing}}^{+}(Y_{\sigma},R)[m]\]
for $+=\text{even,odd}$. Since $H^i(X_{\sigma},R)$ and $H^i(Y_{\sigma},R)$ are torsion free groups for $i=0,1,4$, and vanish for $i<0, i > 4$, this produces the isomorphisms claimed in the statement. The second statement follows from the universal coefficients theorem.
\end{proof}

\begin{cor}\label{bounding}
Let $G$ be a finite group, and $X$ be a complex projective surface with faithful $G$-action $G \curvearrowright X$. If the quotient $X/G$ is also a complex projective manifold and the canonical map $X \to X/G$ induces an equivalence of bounded derived categories, then the vanishing
\[
G^{\text{ab}}[p]=0
\]
holds for any prime $p>7$.
\end{cor}

\begin{conj}\label{conj1}
Let $G$ be a finite group, and $X$ be a complex projective $d$-fold with faithful $G$-action $G \curvearrowright X$. If the quotient $X/G$ is also a complex projective manifold and the canonical map $X \to X/G$ induces an equivalence of bounded derived categories, then the vanishing
\[
G^{\text{ab}}[p]=0
\]
holds for any prime $p>3d+1$.
\end{conj}
\begin{example}\label{example}
In \cite{GrossPopescu}, Gross and Popescu constructed a simply-connected complex projective manifold $Y$ which is an abelian surface fibration to $\PP^1$ with exactly 64 sections with a group action $H=(\Z/8\Z)^2 \curvearrowright Y$, and conjectured that the quotient $Y/H$ should be the \textit{mirror of mirror} of $Y$. Homological mirror symmetry would therefore predict that $D^b(Y) \simeq D^b(Y/H)$, and actually the equivalence is proved by Schnell \cite{fundamentalisnotderived}. Since $Y$ is simply-connected, we know isomorphisms $\pi(Y)=0$ and $\pi(Y/H)=H$, by universal coefficients theorem we know $\text{Tors}(H^2(Y,\Z)) \neq \text{Tors}(H^2(Y/H,\Z))$, but for any $(m,2)=1$, we have $\text{Tors}(H^2(Y,\Z))[m] \simeq  \text{Tors}(H^2(Y/H,\Z))[m]=0$.
\end{example}
\section{The application to algebraic geometry of positive characteristic}
\begin{defn}
For a smooth proper variety $X$ over a perfect field $k$ of characteristic $p$, the variety $X$ is said to be $\textit{ordinary}$ if $H^i(X,d\Omega^j_{X/k})=0$ for any $i$ and $j>0$. $X$ is said to be $\textit{Hodge}$-$\textit{Witt}$ if the $W$-module $H^i(X,W\Omega^j_{X/k})$ is finitely generated over $W$ for any $i$ and $j$.
\end{defn}
\begin{prop}(\cite[Theorem~4.1.3]{Joshi})\label{threeJoshi}
For a smooth proper variety $X$, the following are equivalent:
\begin{itemize}
    \item[] \qquad\qquad\qquad\qquad\qquad\qquad\qquad\quad(1) $X$ is ordinary, 
   { \item[] \qquad\qquad\qquad\qquad\qquad\qquad\qquad(2) $X\times X$ is ordinary,}
   { \item[] \qquad\qquad\qquad\qquad\qquad\qquad\qquad(3) $X\times X$ is Hodge-Witt.}
\end{itemize}
\end{prop}
\begin{proof}
As a product of ordinary varieties is ordinary \cite{Illusie}, thus $(1)\Longrightarrow(2)$. It is well known that an ordinary variety is Hodge-Witt variety, thus $(2) \Longrightarrow (3)$. By \cite[Prop~\rm I\hspace{-.15em}I\hspace{-.15em}I 7.2(ii)]{Ekedahl} if $X \times X$ is Hodge-Witt then $X$ is ordinary, thus we obtain $(3) \Longrightarrow (1)$.
\end{proof}

In \cite{Rulling}, for a perfect field $k$ of characteristic $p$ and $k$-scheme $S$, R\"{u}lling-Chatzistamatiou have constructed a $W$-linear functor (see \cite[Theorem~1]{Reciprocity16}):
\begin{eqnarray}\label{RULLING}
C_{S,W} &\to& (W\sO_S\text{-module})\\
f:X\to S &\mapsto & \bigoplus_{i,j\in\Z} R^if_*W\Omega_{X}^j.\nonumber
\end{eqnarray}
In the case $S=\Spec k$, the functor induces the following additive functor (see \cite[Lemma~3.5.5]{Rulling} and \cite[Appendix B]{Reciprocity16})
\begin{eqnarray}\label{monsterhouse}
\Gamma_{HW}^r:C_{k,W}& \to& (W\text{-module})\\
X &\mapsto & \bigoplus_{j-i=r} H^i(X,W\Omega_{X}^j).\nonumber
\end{eqnarray}
\begin{thm}\label{4.3}
Let $X$ and $Y$ be smooth projective varieties over a perfect field $k$ with characteristic $p$. Suppose $X,Y\in \Sm\Proj^{\leq d}_{(e)}(k)$ for some $e$ and $d$, and $p>2d+e+1$. We assume that there is a $k$-linear fully faithful triangulated functor $F:D^b(X) \to D^b(Y)$. Then the following holds
\begin{itemize}

    \item[(1)] For an integer $r$, there is a split injective morphism of $W$-module 
    \[
    \bigoplus_{j-i=r}H^i(X,W\Omega^j_X) \hookrightarrow \bigoplus_{j-i=r}H^i(Y,W\Omega^j_Y).
    \]
    \item[(2)] For an integer $r$, if $F$ is an equivalence, then there is an isomorphism of $W$-module 
    \[
    \bigoplus_{j-i=r}H^i(X,W\Omega^j_X) \simeq \bigoplus_{j-i=r}H^i(Y,W\Omega^j_Y).
    \]
    \item[(3)] If $Y$ is Hodge-Witt, then $X$ is also Hodge-Witt.
    \item[(4)] If $Y$ is ordinary, then $X$ is also ordinary.
\end{itemize}
\end{thm}
\begin{proof}
$(1)$. By \cite[Lemma 3.2]{M1}, there is an split injective
\[
X \hookrightarrow Y
\]
in $\KM(k)$. For simplicity of notation we write $R$ instead of ${\Z[\frac{1}{(2d+e+1)!}]}$. The $R$-linear functor
\[
\omega_R:\KM_{(e)}^{\leq d}(k)_R \to {C_{k,R}}
\]
induces the split injective
\[
\theta:X_R \hookrightarrow Y_R
\]
in $C_{k,R}$. Since $p>2d+e+1$, the Witt ring $W$ is $R$-algebra, thus the $R$-linear functor \eqref{monsterhouse} induces the split injective map of $W$-module
\[
\bigoplus_{j-i=r}H^i(X,W\Omega^j_X) \hookrightarrow \bigoplus_{j-i=r}H^i(Y,W\Omega^j_Y).
\]
$(2)$. Since there is an equivalence of dg-categories 
\[
\text{perf}_{dg}(X) \simeq \text{perf}_{dg}(Y),
\]
we have an isomorphism of non-commutative motives
\[
U(X)_R\simeq U(Y)_R.
\]
By the $R$-linear functor $\omega_R:\KM_{(e)}^{\leq d}(k)_R \to {C_{k,R}}$, we have an isomorphism 
\[
X_R \simeq Y_R
\]
in $C_{k,R}$. The functor \eqref{monsterhouse} induces an isomorphism of $W$-module 
    \[
    \bigoplus_{j-i=r}H^i(X,W\Omega^j_X) \simeq \bigoplus_{j-i=r}H^i(Y,W\Omega^j_Y).
    \]
$(3)$. By $(1)$, we obtain the claim.
$(4)$. By \cite[Corollary 3.4]{M1}, there is a split injective morphism 
\[
M(X)_{R}(d_Y)[2d_Y] \hookrightarrow \bigoplus_{i=0}^{d_X+d_Y}M(Y)_{R}(i)[2i]
\]
in $\DM^\eff(k,R)$. By the tensor structure in $\DM^\eff(k,R)$, we have a split injective map
\[
h(X\times X) \hookrightarrow \bigoplus_{i=0}^{d_X+d_Y}\bigoplus_{j=0}^{d_X+d_Y} h(Y\times Y)\otimes T^{i+j}
\]
in $\Chow(k)_R$. Thus we have a split injective
\[
(X\times X)_R \hookrightarrow \bigoplus_{i=0}^{d_X+d_Y}\bigoplus_{j=0}^{d_X+d_Y} (Y\times Y)_R
\]
in $C_{k,R}$. Thus we obtain the claim.
\end{proof}

\section{Descent of semi-orthogonal decomposition}

In this section, we shall show that the property of the semi-orthogonal decomposition of derived categories is stable with respect to field extensions. Theorem~\ref{5.2} follows from Proposition \ref{propprop} and Lemma \ref{RUNA} below.

\begin{thm}\label{5.2}
Take smooth projective varieties $X$ and $Y_1$,...,$Y_j$ over $K$ and Azumaya algebras $\alpha_e$ on $Y_e$ for each $e$. If there is a semi-orthogonal decomposition 
\[
D^b(X_{\ol{K}}) \overset{\text{SOD}}{\simeq} <D^b(Y_{1,\ol{K}}, \alpha_{1,\ol{K}}),....,D^b(Y_{j,\ol{K}}, \alpha_{j,\ol{K}})>,
\]
then there is a finite extension $L/K$ such that there is a semi-orthogonal decomposition 
\[
D^b(X_{L}) \overset{\text{SOD}}{\simeq} <D^b(Y_{1,L}, \alpha_{1,L}),....,D^b(Y_{j,L}, \alpha_{j,L})>
\]
\end{thm}

\subsection{} In this section, we shall recall the twisted sheaves theory and twisted Fourier-Mukai transformation from \cite{Azumayabook} and \cite{twistedFourierMukai}. For an noetherian scheme $X$, let $\sR$ be a $\sO_X$-algebra. Assume that $\sR$ is finitely generated projective $\sO_X$-module and that the morphism of $\sO_X$-modules:
\[
\sR\otimes_{\sO_X} \sR^{\text{op}} \to \mathcal{E}nd_{\sO_X}(\sR)
\]
is an isomorphism, where $\sR^{\text{op}}$ is the opposite algebra of $\sR$. Then $\sR$ is called an \textit{Azumaya} algebra on $X$. We say that two Azumaya algebra $\sR$ and $\sP$ on $X$ are \textit{Morita equivalent}, written $\sR\sim_{M} \sP$, if there are finitely generated projective $\sO$-modules $\sF$ and $\sF'$ such that
\[
\sR\otimes_{\sO_X} \mathcal{E}nd_{\sO_X}(\sF) \simeq \sP \otimes_{\sO_X} \mathcal{E}nd_{\sO_X}(\sF').
\]
We consider the class $\{\text{Azumaya algebra on }X\}$ of Azumaya algebras on $X$. The Brauer group $Br(X)$ is the group given by:
\[
Br(X)=\{\text{Azumaya algebra on }X\}/ \sim_{\text{M}}
\]
with a group structure given by the tensor product over $\sO_X$. Gabber and de Jong showed that if $X$ is  quasi-compact and separated then $Br(X)$ is isomorphic to the torsion in the \'etale cohomology group $H^2_\et(X,\mathbb{G}_m)_{\text{tor}}$, \cite[Theorem]{Jong}. We identify an Azumaya algebra $\sR$ with the corresponding element $\alpha \in H^2_\et(X,\mathbb{G}_m)_{\text{tor}}$. For an Azumaya algebra $\alpha$, there is an \'etale covering $\{U_i \to X\}$ such that $\alpha$ is represented by a \v{C}ech cocycle $\alpha_{ijk}\in \Gamma(U_i \times_X U_j \times_X U_k, \mathbb{G}_m)$. An $\alpha$-\textit{twisted sheaf} is given by a system $(\mathcal{M}_{i},\phi_{ij})$ where each $\mathcal{M}_i$ is quasi-coherent $\sO_{U_i}$-module on $U_i$, and where $\phi_{ij}:\mathcal{M}_i \otimes_{\sO_{U_i}} \sO_{ij} \to \mathcal{M}_j\otimes_{\sO_{U_j}} \sO_{U_{ij}}$ are isomorphism and satisfies
\[
\phi_{jk} \circ \phi_{ij}|_{{U_{ijk}}} = \alpha_{ijk}\phi_{ik}|_{{U_{ijk}}}
\]
over $U_{ijk}$. Let us denote $\text{qCoh}(X,\alpha)$ the Abelian category of $\alpha$-{twisted sheaves} and $\text{Coh}(X,\alpha)$ the Abelian category of $\alpha$-{twisted coherent sheaves}, and we write its bounded derived category by $D^b(X,\alpha)$. We shall recall Canonaco-Stellar's result. For two Azumaya algebra $\alpha$ and $\beta$ on $X$, take an \'etale covering $\{U_i \to X\}$ such that $\alpha$ and $\beta$ can be represented by $\alpha_{ijk}\in \Gamma(U_i \times_X U_j \times_X U_k, \mathbb{G}_m)$ and $\beta_{ijk}\in \Gamma(U_i \times_X U_j \times_X U_k, \mathbb{G}_m)$. %
For an $\alpha$-twisted sheaf $\mathcal{M}_a=(\mathcal{M}_{a,i},\phi_{a,ij})$ and a $\beta$-twisted sheaf $\mathcal{M}_b=(\mathcal{M}_{b,i},\phi_{b,ij})$, let us denote $\mathcal{M}_a\otimes \mathcal{M}_b$ the $\alpha\beta$-twisted sheaf given by the system: 
\[
(\mathcal{M}_{a,i} \otimes_{\sO_{U_i}}\mathcal{M}_{b,i}, \phi_{a,ij}\otimes_{\sO_{U_{ij}}} \phi_{b,ij}).
\]
For an $\alpha$-twisted sheaf $P=(\mathcal{M}_{a,i},\phi_{a,ij})$, we say $P$ is \textit{projective} if each $\mathcal{M}_{a,i}$ is a locally free $\sO_{U_i}$-module. For a projective bounded complex $P^\bullet$ of $\alpha$-twisted sheaves and a complex $M^\bullet$ of $\beta$-twisted sheaves, let us denote $P^\bullet \otimes^{\mathbb{L}} M^\bullet$ the total complex $\text{Tot}(P^\bullet \otimes M^*)$. There is a functor
\[
\text{Comp}^{b}(\text{Coh}(X,\beta))\overset{P^\bullet \otimes^{\mathbb{L}}-}{\longrightarrow} \text{Comp}^{b}(\text{Coh}(X,\alpha\beta)).
\]
Moreover, this extends to the derived functor:
\[
D^b(X,\beta) \overset{P^\bullet \otimes^{\mathbb{L}}-}{\longrightarrow} D^b(X,\alpha\beta).
\]
For a flat morphism $f:Y \to X$ of noetherian separated schemes and an Azumaya algebra $\alpha$ on $X$, an $\alpha$-twisted sheaf $(\mathcal{M}_i,\phi_{ij})$, we shall study the pull back functor. The pull back $f^*(\mathcal{N}_i,\phi_{ij})$ is given by $(f^*\mathcal{M}_i,f^*\phi_{ij})$ and it is a $f^*\alpha$-twisted sheaf, and there are functors $f^*:\text{qCoh}(X,\alpha) \to \text{qCoh}(Y,f^*\alpha)$ and $f^*:\text{Coh}(X,\alpha) \to \text{Coh}(Y,f^*\alpha)$. 

For a flat projective morphism $f:Y \to X$ of noetherian and separated schemes and an Azumaya algebra $\alpha$ on $X$, choose an \'etale covering $\{U_i \to X\}$ and a \v{C}ech cocyle $\alpha_{ijk}\in \Gamma(U_i\times_X U_j\times_X U_k,\mathbb{G}_m)$ representing $\alpha$. Then $f^*\alpha$ is represented by the \'etale covering $\{f^{-1}U_i \to Y\}$ and $f^{*}\alpha_{ijk}\in \Gamma(f^{-1}U_i\times_Y f^{-1}U_j\times_Y f^{-1}U_k,\mathbb{G}_m)$. For a $f^*\alpha$-twisted sheaf $(\mathcal{M}_i,\phi_{ij})$, the push forward $f_*(\mathcal{N}_i,\phi_{ij})$ is given by $(f_*\mathcal{N}_i, f_\sharp\phi_{ij})$ where $f_\sharp\phi_{ij}: f_*\mathcal{N}_i\otimes_{U_i}U_{ij} \simeq  f_*\mathcal{N}_j\otimes_{U_j}U_{ij}$ is the isomorphism which fits into the following commutative diagram:
\[
\xymatrix{
f_*\bigl(\mathcal{N}_i \otimes_{\sO_{f^{-1}U_{i}}}\sO_{f^{-1}U_{ij}}\bigr) \ar[rr]_-{\text{proj.form}}^-{\simeq} \ar[d]_{f_*\phi_{ij}}^{\simeq}& &f_*\mathcal{N}_i \otimes_{\sO_{U_{i}}}\sO_{U_{ij}} \ar[d]^-{f_\sharp\phi_{ij}}\\
f_*\bigl(\mathcal{N}_j \otimes_{\sO_{f^{-1}U_{j}}}\sO_{f^{-1}U_{ij}}\bigr) \ar[rr]_-{\text{proj.form}}^-{\simeq} && f_*\mathcal{N}_j \otimes_{\sO_{U_{j}}}\sO_{U_{ij}}
}
\]

\begin{lemma} \label{lemmafFfG}
For a flat projective morphism $g: V\to U$ and $w\in \Gamma(U,\sO_{U}^*)$ and coherent sheaves $F$ and $G$ on $V$, we assume there is an isomorphism $\phi:F \simeq G$. Then $g^*w$ induces an isomorphism $g^*w\cdot\phi:F\simeq G$ and satisfies the following
\[
g_*(g^*w\cdot \phi)\simeq w\cdot g_*\phi
\]
as isomorphisms from $f_*F$ to $f_*G$.
\end{lemma}

\begin{proof}
We shall regard $g^*w$ as an isomorphism $g^*w:\sO_V\simeq \sO_V$. By the projection formula, there is a commutative diagram:
\[
\xymatrix{
g_*F \ar[d]_-{g_*(g^*w\cdot \phi)}^-{\simeq} \ar[r]^-{\simeq} & g_*(F\otimes_{\sO_V}\sO_V) \ar[d]^-{g_*(\phi\otimes g^*w)}_-{\simeq} \ar[r]_-{\text{proj.form}}^-{\simeq} & g_*(F)\otimes_{\sO_U}\sO_U \ar[d]_-{\simeq}^-{g_*\phi\otimes w}\\
g_*G \ar[r]_-{\simeq} & g_*(G\otimes_{\sO_V}\sO_V) \ar[r]_-{\text{proj.form}}^-{\simeq}&  g_*(G)\otimes_{\sO_U}\sO_U
}
\]
Thus we obtain the claim. 
\end{proof}
Applying Lemma~\ref{lemmafFfG} for a flat projective morphism, $f^{-1}U_i \to V_i$ and an isomorphism $\phi_{jk} \circ \phi_{ij}|_{{f^{-1}U_{ijk}}}: \mathcal{N}_{i}\otimes_{\sO_{f^{-1}U_i}}\sO_{f^{-1}U_{ijk}} \simeq \mathcal{N}_{k}\otimes_{\sO_{f^{-1}U_k}}\sO_{f^{-1}U_{ijk}}$, we know $f_*(f^*\alpha_{ijk}\phi_{ik})=\alpha_{ijk}f_{\sharp}\phi_{ik}$ and we obtain
\[
f_\sharp\phi_{jk} \circ f_\sharp\phi_{ij} =\alpha_{ijk} f_\sharp\phi_{ik}.
\]
Thus we know that $(f_*\mathcal{N}_i, f_\sharp\phi_{ij})$ is a $\alpha$-twisted sheaf and there are functors $f_*:\text{qCoh}(Y,f^*\alpha) \to \text{qCoh}(X,\alpha)$. If $(\mathcal{N}_i,\phi_{ij})$ is a coherent, then $(f_*\mathcal{N}_i, f_\sharp\phi_{ij})$ is also coherent since $f_*\mathcal{N}_i$ is a coherent sheaf on $U_i$, thus there is a functor $f_*:\text{Coh}(Y,f^*\alpha) \to \text{Coh}(X,\alpha)$. This functor induces the following
\[
Rf_*:D^b(Y,f^*\alpha) \to D^b(X,\alpha),
\]
where we need to check that $f_*$ is left exact, and each $R^if_*(\mathcal{N}_i,\phi_{ij})$ is coherent for any coherent $f^*\alpha$-twisted sheaf. Let us prove this. For a flat morphism $g:U \to X$, we consider the following Cartesian diagram:
\[
\xymatrix{
U\times_X Y \ar[r]^{\tilde{g}} \ar[d]_{\tilde{f}} & Y \ar[d]^{f} \\
U \ar[r]_{g} & X
}
\]
It is easy to check that $g^* \circ f_* \simeq \tilde{f}_*\circ \tilde{g}^*$ as functors $\text{Coh}(Y,f^*\alpha) \to \text{Coh}(U,g^*\alpha)$. We denote the \'etale map $U_i \to X$ by $q_i$. For any $i$, we consider the following Cartesian diagram:
\[
\xymatrix{
f^{-1}U_i \ar[r]^{\tilde{q}_i} \ar[d]_{\tilde{f}} & Y \ar[d]^{f} \\
U_i \ar[r]_{q_i}& X
}
\]
Since $q_i$ is an \'etale morphism, it is enough to prove that $\tilde{f}_*:\text{Coh}(f^{-1}U_i,\tilde{f}^*q_i^*\alpha) \to \text{Coh}(U_i,q_i^*\alpha)$ is left exact, and each $R^i\tilde{f}_*$ preserves coherentness. Since $\alpha$ is trivial on $U_i$ we know that every $\tilde{f}^*q_i^*\alpha$-twisted coherent sheaf is a usual coherent sheaf on $f^{-1}U_i$. Thus we obtain the claim. Thanks to this, we obtain the following functor
\[
Rf_*:D^b(Y,f^*\alpha) \to D^b(X,\alpha).
\]
For a $f^*\alpha$-twisted sheaf $(\mathcal{M}_i,\phi_{ij})$, it is easy to see that there is an isomorphism
\begin{equation}\label{YAOYA}
R^if_*(\mathcal{M}_i,\phi_{ij})=(R^if_*\mathcal{M}_i,R^if_\sharp\phi_{ij}),
\end{equation} where $R^if_\sharp \phi_{ij}$ is the morphism fits into the following diagram
\[
\xymatrix{
R^if_*\bigl(\mathcal{N}_i \otimes_{\sO_{f^{-1}U_{i}}}\sO_{f^{-1}U_{ij}}\bigr) \ar[rr]_-{\text{proj.form}}^-{\simeq} \ar[d]_{R^if_*\phi_{ij}}^{\simeq}& &R^if_*\mathcal{N}_i \otimes_{\sO_{U_{i}}}\sO_{U_{ij}} \ar[d]^-{R^if_\sharp\phi_{ij}}\\
R^if_*\bigl(\mathcal{N}_j \otimes_{\sO_{f^{-1}U_{j}}}\sO_{f^{-1}U_{ij}}\bigr) \ar[rr]_-{\text{proj.form}}^-{\simeq} && R^if_*\mathcal{N}_j \otimes_{\sO_{U_{j}}}\sO_{U_{ij}}
}
\]

Here is a big theorem proved by Canonaco-Stellari \cite{twistedFourierMukai}.
\begin{thm}[Canonaco-Stellari, Theorem~1.1 \cite{twistedFourierMukai}]\label{twistedreepresentation}
Let $X$ and $Y$ be smooth projective varieties over a field $K$, and $\alpha$ and $\beta$ be Azumaya algebra on $X$ and $Y$, respectively. For a triangulated $K$-linear full functor $F:D^b(X,\alpha) \to D^b(Y,\beta)$, there exist $\sE\in D^b(X\times Y,\alpha^{-1} \boxtimes \beta)$ such that $F$ is equivalent to the following functor
\begin{equation}\label{eqHIYAKE}
D^b(X,\alpha) \overset{q^*}{\to} D^b(X\times Y,q^*\alpha) \overset{\sE\otimes^{\mathbb{L}}-}{\to} D^b(X\times Y, p^*\beta) \overset{Rp_*}{\to} D^b(Y,\beta) 
 \end{equation}
 where $p$ is the projection $X\times Y \to Y$ and $q$ is the projection $X\times Y \to X$, and $\alpha^{-1} \boxtimes \beta =q^*\alpha^{-1}\cdot p^*\beta$.
\end{thm}

Let us prove that the functor \eqref{eqHIYAKE} compatible with any flat base change.

For flat morphisms $f:Y \to X$ and $g:Z \to X$ of noetherian quasi-compact and separated schemes, Azumaya algebras $\alpha$ on $X$. Let us denote $W$ the Cartesian $Z\times_X Y$. 
\[
\xymatrix{
W \ar[r]^{\tilde{g}} \ar[d]_{\tilde{f}} & Y \ar[d]^{f} \\
Z \ar[r]_{g} & X
}
\]
where $\ol{g}$ (resp. $\ol{f}$) is the base of $g$ (resp. $f$) along $f$ (resp. $g$). Consider the following diagram:
\begin{equation}\label{5555}\xymatrix{
D^b(Y,f^*\alpha) \ar[rr]^-{\tilde{g}^*} & & D^b(W,\tilde{g^*}f^*\alpha )\\
D^b(X,\alpha) \ar[rr]_{g^*} \ar[u]^{f^*}& & D^b(Z,g^*\alpha) \ar[u]_{f^*}
}
\end{equation}
We shall obtain the following by the definition of flat pullback of twisted sheaves.
\begin{lemma}
The diagram \ref{5555} commutes.
\end{lemma}

For a flat morphism $f:Y \to X$ of noetherian quasi-compact and separated schemes, Azumaya algebras $\alpha$ and $\beta$ on $X$ and a complex of projective $\alpha$-twisted sheaves $P^\bullet$, we consider the following diagram:
\begin{equation}\label{pullbackdiagram}\xymatrix{
D^b(Y,f^*\beta) \ar[rr]^{f^*P^{\bullet}\otimes^{\mathbb{L}}-} & & D^b(Y,f^*\alpha f^*\beta)\\
D^b(X,\beta) \ar[rr]_{P^{\bullet}\otimes^{\mathbb{L}}-} \ar[u]^{f^*}& & D^b(X,\alpha\beta) \ar[u]_{f^*}
}
\end{equation}
\begin{lemma}\label{5.4}
The diagram \eqref{pullbackdiagram} commutes.
\end{lemma}
\begin{proof}
Let us denote $P^l=(P^l_i,\psi_{ij}^l)$. For a complex $(\mathcal{M}_i\bullet,\phi_{ij}^{\bullet})$ of $\beta$-twisted sheaves on $X$, we know 
\[
    f^*P^\bullet \otimes^{\mathbb{L}} f^*(\mathcal{M}_i\bullet,\phi_{ij}^{\bullet}) = (\text{Tot}(f^*\mathcal{M}_i^\bullet\otimes_{\sO_{f^{-1}U_i}} f^*P^*_i),\text{Tot}(f^*\phi_{ij}^\bullet\otimes_{\sO_{f^{-1}U_{ij}}} f^*\psi_{ij}^{*})).
\]
Since $f$ is flat, we have isomorphisms $\text{Tot}(f^*\mathcal{M}_i^\bullet\otimes_{\sO_{f^{-1}U_i}} f^*P^*_i) \simeq f^*\text{Tot}(\mathcal{M}_i^\bullet\otimes_{\sO_{U_i}} P^*_i)$ and $\text{Tot}(f^*\phi_{ij}^\bullet\otimes_{\sO_{f^{-1}U_{ij}}} f^*\psi_{ij}^{*})\simeq f^*\text{Tot}(\phi_{ij}^\bullet\otimes_{\sO_{U_{ij}}} \psi_{ij}^{*})$. Thus the right hand side of the equation is $f^*(P^\bullet \otimes(\mathcal{M}_i\bullet,\phi_{ij}^{\bullet}))$.
\end{proof}
Given a flat morphism $f:Y \to X$ and projective morphism $g:Z \to X$ of noetherian quasi-compact separated schemes and an Azumaya algebras $\alpha$ on $X$. Let us denote $W$ the Cartesian $Z\times_X Y$. 
\[
\xymatrix{
W \ar[r]^{\tilde{g}} \ar[d]_{\tilde{f}} & Y \ar[d]^{f} \\
Z \ar[r]_{g} & X
}
\]
Consider the following diagram.
\begin{equation}\label{TAKE}\xymatrix{
D^b(W,\tilde{f}^*g^*\alpha) \ar[rr]^{R\tilde{g}_*} & & D^b(Y,f^*\alpha )\\
D^b(Z,g^*\alpha) \ar[rr]_{Rg_*} \ar[u]^{\tilde{f}^*}& & D^b(X,\alpha) \ar[u]_{f^*}
}
\end{equation}
\begin{lemma}\label{5.5}
The diagram \eqref{TAKE} commutes.
\end{lemma}
\begin{proof}
The claim follows from the isomorphism \eqref{YAOYA}.
\end{proof}
\begin{prop}\label{propprop}
Let $X$ and $Y$ be smooth projective varieties over a field $K$, $\alpha$ and $\beta$ be Azumaya algebras on $X$ and $Y$, respectively. For a triangulated $\ol{K}$-linear fully faithful functor $F:D^b(X_{\ol{K}},\alpha_{\ol{K}}) \to D^b(Y_{\ol{K}},\beta_{\ol{K}})$, there is a finite extension $L/K$ such that there exits an $L$-linear fully faithful functor $F_L:D^b(X_L,\alpha_L) \to D^b(Y,\beta_L)$ which fits into the following commutative diagram:
\begin{equation}\label{TAKE2}\xymatrix{
D^b(X_{\ol{K}},\alpha_{\ol{K}}) \ar[rr]^{F} & & D^b(Y_{\ol{K}},\beta_{\ol{K}} )\\
D^b(X_L,\alpha_L) \ar[rr]_{F_L} \ar[u]& & D^b(Y_L,\beta_L) \ar[u]
}
\end{equation}
where vertical maps are given by the base change along $\ol{K}/L$.
\end{prop}

\begin{proof}
By Proposition~\ref{twistedreepresentation}, there is an object $\sE \in D^b(X_{\ol{K}}\times_{\ol{K}} Y_{\ol{K}},\alpha_{\ol{K}}^{-1}\boxtimes\beta_{\ol{K}})$ which represents $F$, i.e. the functor $F$ is isomorphic to the functor
\[
D^b(X_{\ol{K}},\alpha_{\ol{K}}) \overset{\ol{q}^*}{\to} D^b(X_{\ol{K}}\times_{\ol{K}} Y_{\ol{K}},\ol{q}^*\alpha_{\ol{K}}) \overset{\sE\otimes^{\mathbb{L}}-}{\to} D^b(X_{\ol{K}}\times_{\ol{K}} Y_{\ol{K}}, \ol{p}^*\beta_{\ol{K}}) \overset{R\ol{p}_*}{\to} D^b(Y_{\ol{K}},\beta_{\ol{K}}).
\]
We can take a finite extension $L/K$ such that there is an object $\sE_L\in D^b(X_{L}\times_{L} Y_{L},\alpha_{L}^{-1}\boxtimes\beta_L)$ and an isomorphism $f^*\sE_L \simeq \sE$ where $f$ isa natural map $X_{\ol{K}} \times_{\ol{K}}Y_{\ol{K}} \to X_L\times_L Y_L$. 

Consider the following diagram:
\[
\xymatrix{
D^b(X_{\ol{K}},\alpha_{\ol{K}}) \ar[r]^-{\ol{q}^*} & D^b(X_{\ol{K}}\times_{\ol{K}} Y_{\ol{K}},\ol{q}^*\alpha_{\ol{K}}) \ar[r]^-{\sE\otimes^{\mathbb{L}}-} & D^b(X_{\ol{K}}\times_{\ol{K}} Y_{\ol{K}}, \ol{p}^*\beta_{\ol{K}}) \ar[r]^-{R\ol{p}_*} & D^b(Y_{\ol{K}},\beta_{\ol{K}})\\
D^b(X_L,\alpha_L) \ar[r]_-{q_L^*} \ar[u]& D^b(X_L\times_L Y_L,q_L^*\alpha_L) \ar[r]_-{\sE_L \otimes^{\mathbb{L}}-} \ar[u] & D^b(X_L\times_{L} Y_L, p_L^*\beta_L) \ar[r]_-{Rq_L^*} \ar[u] & D^b(Y_L,\beta_L) \ar[u]
}
\]
where each of vertical maps are given by flat base change map along $\Spec \ol{K} \to \Spec L$. Be Lemma~\ref{5.4} and Lemma~\ref{5.5}, the above diagram commutes. We denote $Rq_L^* \circ \sE_L \otimes^{\mathbb{L}}- \circ q_L^*$ by $F_L$. Let us show that $F_L$ is fully faithful. By the base change theorem of coherent sheaves, for objects $a,b\in D^b(X_L,\alpha_L)$, we have a commutative diagram:
\[
\xymatrix{
\Hom_{D^b(X_{\ol{K}},\alpha_{\ol{K}})}(a_{\ol{K}},b_{\ol{K}}) \ar[r]_-{\simeq}^-{F} & \Hom_{D^b(Y_{\ol{K}},\beta_{\ol{K}})}(F(a_{\ol{K}}),F(b_{\ol{K}}))\\
\Hom_{D^b(X_L,\alpha_L)}(a,b)\otimes_L \ol{K} \ar[r]_-{F_L\otimes_L \ol{K}}^-{\simeq} \ar[u]^{\simeq}&\Hom_{D^b(Y_L,\beta_L)}(F_L(a),F_L(b))\otimes_L \ol{K} \ar[u]^{\simeq}
}
\]
Since $\ol{K}$ is a flat over $L$, we obtain the claim.
\end{proof}
\begin{lemma}
Take smooth projective varieties $X$ and $Y_1$,...,$Y_j$ over $K$ and Azumaya algebra $\alpha_e$ on $Y_e$ for each $e$. If there is a semi-orthogonal decomposition 
\[
D^b(X_{\ol{K}}) \overset{\text{SOD}}{\simeq} <D^b(Y_{1,\ol{K}}, \alpha_{1,\ol{K}}),....,D^b(Y_{j,\ol{K}}, \alpha_{j,\ol{K}})>,
\]
then there is a finite extension $L/K$ and a full subcategory $T\hookrightarrow D^b(X_{L})$ such that there is a semi-orthogonal decomposition 
\[
T \overset{\text{SOD}}{\simeq} <D^b(Y_{1,L}, \alpha_{1,L}),....,D^b(Y_{j,L}, \alpha_{j,L})>
\]
\end{lemma}
\begin{proof}
By Proposition~\ref{propprop}, there is a $L$-linear fully faithful functor $F_{e,L}:D^b(Y_{e,L}, \alpha_{e,L}) \hookrightarrow D^b(X_L)$. Let us prove that, for for all ${\displaystyle 1\leq i<l\leq e}$ and all objects $a\in D^b(Y_{l,L}, \alpha_{l,L})$ and $b\in  D^b(Y_{i,L}, \alpha_{i,L})$:
\[
\Hom_{D^b(X_L)}(F_{l,L}(a),F_{i,L}(b))=0
\]
By the base change theorem of coherent sheaves, this equation follows from the equation
\[
\Hom_{D^b(X_{\ol{K}})}((F_{l,L}(a))_{\ol{K}},(F_{i,L}(b))_{\ol{K}})=\Hom_{D^b(X_{\ol{K}})}(F_l(a_{\ol{K}}),F_i(b_{\ol{K}}))=0.
\]
We obtain the claim.
\end{proof}

\subsection{}Let $Q$ be the inverse limit of the inverse system $(Q_\lambda,{u_{\lambda}}_{\mu})$ in the category of schemes. We assume that the transition ${u_{\lambda}}_{\mu}:Q_{\lambda}\to Q_{\mu}$ is finite for any $\lambda < \mu$. Suppose that $Q_\lambda$ is quasi-compact and quasi-separated. In this case, for any coherent sheaf $\sF$ on $\lim Q_\lambda$ there is a $\lambda$ and a coherent sheaf $\sF_{\lambda}$ on $Q_{\lambda}$ such that $\sF \simeq u_{\lambda}^* \sF_{\lambda}$ where $u_{\lambda}$ is the natural map $Q \to Q_{\lambda}$ (see \cite[Corollaire  (8.5.2.4).]{EGA43}).

\begin{defn}
Let $X$ be a quasi-compact and quasi-separated scheme, $S$ be a countable set. For objects $\sE_i\in D^b(X)$ $i\in S$ and a set $\{\sE_i\}_{i\in S}$, we define the sub category $<\sE_i|i \in S>$ for the smallest sub triangulated category containing $E_{i}$ and closed under finite coproducts, shift and cone. We say $\{\sE_i\}_{i\in S}$ is a generating set of $D^b(X)$ iff $D^b(X)=<\sE_i|i \in S>$.
\end{defn}

\begin{lemma}\label{poporo}
Let $f: X \to Y$ be a morphism of projective varieties over a field $k$. We assume $Y$ is smooth over $k$. Consider objects $\sE_i\in D^b(Y)$ $i \in S$. If an object $E\in  D^b(Y)$ is in $<\sE_i|i \in S>$, then the object $Lf^* E$ is in $<Lf^*\sE_i|i \in S>$.
\end{lemma}
\begin{proof}
The claim follows from the fact that the functor $Lf^*$ preserves finite coproducts, shift and cone.
\end{proof}

\begin{lemma}\label{Naruse}
For a projective smooth variety $X$ over a field $k$, we fix a very ample divisor $\sO_X(1)$. Then the set $\{\sO_X(i)| 0 \geq i \geq -m\}$ is a generating set of $D^b(X)$ where $m=\dim_k H^0(X,\sO_X(1))-1$.
\end{lemma}
\begin{proof}
The set $\{\sO_X(-i)|i\in \N\cup \{0\}\}$ is a generating set of $D^b(X)$ (see \cite[Théorème  (2.2.1) (iii)]{EGA31}). Consider the closed immersion $f:X\to \PP^m$ corresponding to the very ample divisor $\sO_X(1)$. It is well known that the set $\{\sO_{\PP^m}, \sO_{\PP^m}(-1),..., \sO_{\PP^m}(-m)\}$ is a generating set of $D^b(\PP^m)$. For any natural number $i$, we have an isomorphism in $D^b(X)$: \[
Lf^*\sO_{\PP^m}(-i) \simeq f^*\sO_{\PP^m}(-i) \simeq \sO_X(-i).
\]
Since $<\sO_{\PP^m}, \sO_{\PP^m}(-1),..., \sO_{\PP^m}(-m)>= D^b(\PP^m)$, for any $l > m$, $\sO_{\PP^m}(-l)\in <\sO_{\PP^m}, \sO_{\PP^m}(-1),..., \sO_{\PP^m}(-m)>$. Thus we obtain the following:
\[
\sO_X(-l) \in <\sO_X,\sO_X(-1),..,\sO_X(-m)>.
\]
Thus by Lemma~\ref{poporo} we obtain the claim.
\end{proof}

For a projective variety $X$ over a field $K$ and a finite extension $L/K$, we write $X_L=X\otimes_K L$ and we denote the morphism $X_{\ol{K}} \to X_{L}$ by $f_L$.

\begin{lemma}\label{megane}
For a morphism $r: \sF \to \sG \in D^b(X_{\ol{K}})$, there is a finite extension $M/K$ and a morphism $r_M:\sF_M \to \sG_M \in D^b(X_M)$ such that $r\simeq f_M^*r_M$.
\end{lemma}
\begin{proof}
There is a finite extension $L/K$ and objects $\sF_L,\sG_L \in D^b(X_L)$ such that $f_L^*\sF_L \simeq \sF$ and $f_L^*\sG_L \simeq \sG$. Suppose that $Hom_{D^b(X_{\ol{K}})}(\sF,\sG)\simeq \ol{K}^{\oplus m}$ for some $m \in \N$. By the flat base change theorem for coherent cohomology, there is an isomorphism of $L$-vector spaces:
\[
Hom_{D^b(X_{L})}(\sF_L,\sG_L)\simeq L^{\oplus m}.
\]
For a vector $r\in Hom_{D^b(X_{\ol{K}})}(\sF,\sG)= \ol{K}^{\oplus m}$, there is finite extension $N/K$ such that $r \in N^{\oplus m}$. Choose a finite extension $M/K$ containing $N$ and $L$. If we take $r_M\in Hom_{D^b(X_M)}(\sF_M,\sG_M)\simeq M^{\oplus m}$ corresponding to $r \in M^{\oplus m}$, we obtain the claim.
\end{proof}

\begin{lemma}\label{RUNA}
Consider a smooth projective variety $X$ over a field $K$, and full subcategories $E_i: T_i {\hookrightarrow} D^b(X)$ for $1 \leq i \leq j$. We assume that $D^b(X_{\ol{K}})$ is generated by $f_K^*E_i(T_i)$ i.e, it is the smallest subcategory of $D^b(X_{\ol{K}})$ which contains all of the objects in $f_K^*E_i(T_i)$ for any $1\leq i \leq e$ and is closed under finite coproducts, cone and shifts, where $f_K$ is the natural map $X_{\ol{K}} \to X$. Then there is a finite extension $L/K$ such that $D^b(X_{L})$ is generated by $f_{L/K}^*E_i(T_i)$, where $f_{L/K}$ is the natural map $X_{L} \to X$.
\end{lemma}
\begin{proof}
For an object $G$ in $D^b(X_{\ol{K}})$ we consider the following condition $(1)$:
\begin{itemize}
    \item[($1$):]{There is a finite extension $M/K$ and an object $G_M \in D^b(X_M)$ such that $f_M^* G_M\simeq G$, and $G_M\in <f_{M/K}^*E_{1}(T_1),f_{M/K}^*E_{2}(T_2),...,f_{M/K}^*E_{n}(T_e)>$.}  
\end{itemize}
All objects in $f_K^*E_i(T_i)$ satisfy $(1)$. If an object $G \in D^b(X_{\ol{K}})$ satisfies $(1)$, then $G[l]$ satisfies $(1)$ for any $l\in \Z$. If objects $G_1,..,G_l \in D^b(X_{\ol{K}})$ satisfy $(1)$, then $\bigoplus G_i$ satisfies $(1)$. For a distinguished triangle in $D^b(X_{\ol{K}})$
\[
G_1 \to G_2 \to G_3 \overset{+}{\to} G_1[1],
\]
    if $G_1$ satisfies $(1)$ and $G_2$ satisfies $(1)$, then by Lemma~\ref{megane} $G_3$ satisfies $(1)$. Thus the condition $(1)$ is closed under any finite coproducts, cone and shift. We fix an immersion $X \hookrightarrow \PP^m_K$. For any finite extension $M/K$, there is a Cartesian diagram:
    \[
    \xymatrix{
    X_{\ol{K}} \ar[d]_{f_M}\ar@{^{(}-_>}[r]^-{\ol{i}} & \PP^m_{\ol{K}}\ar[d]\\
    X_M \ar[d]_{f_{M/K}} \ar@{^{(}-_>}[r]^-{i_M}& \PP^m_K\ar[d]    \\
    X \ar@{^{(}-_>}[r]^-{i}& \PP^m_K
    }
    \]
Let us denote $\sO_{\ol{X}}(1)$ the very ample divisor given by $\ol{i}$ and $\sO_{X_M}(1)$ the very ample divisor given by ${i_M}$. We note that there is a natural isomorphism $\sO_{\ol{X}}(i) \simeq f_M^*\sO_{X_M}(i)$ for any $i\in\Z$. For any $m\geq l \geq 0$, $\sO(X_{\ol{K}})(-l)$ satisfies the condition $(1)$, there is a finite extension $M_l/K$ and and an object $H_l$ such that $H_l$ is in $<f_{M_l/K}^*E_1(T_1),...,f_{M_l/K}^*E_n(T_n)>$ and $\sO(X_{\ol{K}})(-l) \simeq f_{M_l}^*H_l$. Since $f_{M_l}$ is faithfully flat, we know $H_l \simeq \sO_{M_{l}}(-l)$. %
Choose a finite extension $L$ containing all $M_l$ over $K$ then $\sO_L(-l) \in <f_{L/K}^*E_1(T_1),...,f_{L/K}^*E_n(T_n)>$ for any $m\geq l\geq 0$. Since we know $D^b(X_L) =<\sO_{X_L},\sO_{X_L}(-1),...,\sO_{X_L}(-m)>$, we have $D^b(X_L)= <f_{L/K}^*E_1(T_1),...,f_{L/K}^*E_n(T_n)>$.
\end{proof}

Theorem~\ref{5.2} follows from Proposition \ref{propprop} and Lemma \ref{RUNA}. We note the following Lemma.

\begin{lemma}\label{5.15}
For an extension of fields $L/K$ and a projective variety $X/K$, given full subcategories $E_i:T_j \hookrightarrow D^b(X)$ for $1 \leq j \leq i$ such that there is a semi-orthogonal decomposition
\[
D^b(X) \overset{\text{SOD}}{\simeq} <E_1(T_1),...,E_i(T_i)>
\]
then there is a semi-orthogonal decomposition
\[
D^b(X_L) \overset{\text{SOD}}{\simeq} <f_{L/K}^*E_1(T_1),...,f_{L/K}^*E_i(T_i)>,
\]
where $f_{L/K}$ is a natural morphism $X_L\to X$.
\end{lemma}
\begin{proof}
There is some $m\in\N$ such that $D^b(X)$ is generated by $\sO_X(-l)$ for $0 \leq l \leq m$ and also $D^b(X_L)$ is generated by $\sO_{X_L}(-l)$ for $0 \leq l \leq m$. Since each $\sO_X(-l)$ is in $<E_1(T_1),...,E_i(T_i)>$, the claim follows from Lemma~\ref{poporo}.
\end{proof}
\section{The application to number theory}
\subsection{} In this section, we first recall the result of Fontaine-Messing on $p$-adic Hodge theory over unramified base, and the result of Achinger on ordinary reduction. 

We let $k$ be a finite field of characteristic $p>0$ and we write $W$ for the Witt ring of $k$ and $K$ for the fraction field of $W$. For a smooth proper variety $X$ over $K$ of good reduction, we write $\mathfrak{X}$ for a smooth proper integral model of $X$ over $W$, and $X_k$ for the special fiber of $\mathfrak{X}$. If $i<p-1$, in \cite{FM87} Fontaine-Messing proved that there is an isomorphism of $W$-modules (see \cite[Theorem A]{FM87} or \cite[Theorem 0.9]{Min21})
\begin{equation}\label{fontaine-Messing}
H^i_{\text{crys}}(X_k/W) \simeq H^i_{\acute{e}t}(X_{\widehat{\ol{K}}},\Z_p) \otimes_{\Z_p} W. 
\end{equation}
where $\widehat{\ol{K}}$ is the completion of $\ol{K}$.

We let $F$ be a complete discrete valuation field of mixed characteristic $(0,p)$, and we write $k$ be the residue field of $\sO$. For a smooth proper variety $\mathfrak{X}$ over $\sO$, we write $X_k$ for the special fiber and $X_{\widehat{\ol{F}}}$ for the geometric fiber over $\widehat{\ol{F}}$ which is the completion of $\ol{F}$ endowed with its unique absolute value extending the given absolute value $|\cdot|$ on $F$. In \cite{AchingerArxiv} Achinger proved the following theorem.
\begin{thm}[Achinger, Proposition 6.7 \cite{AchingerArxiv}]
Consider the following conditions:
\begin{itemize}
    \item[(1)] $X_k$ is ordinary.
    \item[(2)] The $p$-adic Galois representations $H^i(X_{\widehat{\ol{F}}}, \Q_p)$ are ordinary for all $i$.
\end{itemize}
Then $(1) \Longrightarrow (2)$. Moreover, $(2) \Longrightarrow (1)$ if $H^*_{\text{crys}}(X_k/W)$ is a free $W$-module and $H^i(\mathfrak{X},\Omega^j_{\mathfrak{X}/{\sO}})$ are free $\sO$-modules for all $i,j$.
\end{thm}

\subsection{} For a prime number $l\in k^*$, Ivorra, \cite[(128)]{Ivorram}, and later Ayoub, \cite[functor (55)]{Ayoubladic} constructed the $l$-adic realization functor of motives:
\[
\Gamma_{k,l}:\DM^{\eff}_{gm}(k) \to D^{\et}(\Spec k,\Z_l), \text{~~~} [\pi:X \to k] \mapsto R\pi_*\pi^*\Z_l.
\]
Choose an algebraic closure $\ol{k}/k$. We set a functor
\[
\Gamma_{G_k.l}^{pre}:\DM^\eff_{gm}(k) \to \DM^\eff_{gm}(\ol{k}) \overset{\Gamma_{\ol{k},l}}{\to} D^{\et}(\Spec \ol{k},\Z_l),
\]
where the first functor is the natural base change along the finite extension $\ol{k}/K$. This functor sends $[\pi:X\to \Spec k]$ to $R\ol{\pi}_*\ol{\pi}^* \mathbb{Z}_l$, where $\ol{\pi}$ is a morphism $X\times_k \ol{k} \to \Spec \ol{k}$ comes from the base change of $\pi$ along the finite extension $\ol{k}/K$. For a motive $N\in \DM^\eff_{gm}(k)$ and a morphism $a:N \to M$ in $\DM^\eff_{gm}(k)$, $\Gamma_{G_k.l}^{pre}(N)$ and $\Gamma_{G_k.l}^{pre}(a)$ are equipped with an action by $G_k$. Thus we know that there is a functor $\Gamma_{G_k,l}:\DM^\eff_{\gm}(k) \to D(G_k \curvearrowright \Z_l\text{-mod})$ such that the functor $\Gamma_{G_k.l}^{pre}:\DM^\eff_{\text{gm}}(k) \to \DM^\eff_{\text{gm}}(\ol{k}) \to D^\et(\Spec \ol{k}, \Z_l)$ factors though it:
\[
\xymatrix{
\DM^{\eff}_{gm}(k) \ar[d]_{\Gamma_{G_k,l}} \ar[r] &\DM^{\eff}_{gm}(\ol{k}) \ar[r]^-{\Gamma_{\ol{k},l}} & D^{\et}(\Spec \ol{k},\Z_l) \\
D(G_k \curvearrowright \Z_l\text{-mod}) \ar[rru]_-{\text{forgetful}} & & 
}
\]
For a smooth projective variety $X$ over $k$, we have an isomorphism $H^i(\Gamma_{G_k,l}M(X))\simeq H^i(X_{\ol{k}},\Z_l)$. Also we know, for a natural number $j$, there is an isomorphism $H^i(\Gamma_{G_k,l}M(X)(j)[2j])\simeq H^{i-2j}(X_{\ol{k}},\Z_l)(-j)$. %
It is well known that $\{2n\}_n$ is a cohomological dimension of $\Gamma_{G_k,l}$ (see \cite[chapter VI, Theorem 1.1]{Milneetale}). For a natural number $m>6d$, we set 
\[
\Gamma_{G_{k},l,m}^j(X) = \left\{
\begin{array}{ll}
H^m(\Gamma_{G_k,l}M(X)(j)[2j])\simeq H^{m-2j}(X_{\ol{k}},\Z_l)(-j) & j>2d \\
~~~0 & j\leq 2d
\end{array}
\right.
\]
If $\dim X \leq d$, for $m >6d$, then we have an isomorphism
\[
\Gamma_{G_k,l,m}(X) = \bigoplus_{j>2d} \Gamma_{G_k,l,m}^j(X) \simeq \bigoplus_{j>2d} H^{m-2j}_{\et}(X_{\ol{k}},\Z_l)(-j) 
\]
since $H^j(X_\sigma^\text{an},R)=0$ for any $j>2d$, and by Proposition~\ref{2} we know $\Gamma_{G_k,l,m}$ is a $d$-Manin invariant. We note that if we take $m=6d+1$ then we obtain isomorphism of $\Z_l$-module with $G_k$ action:
\begin{equation}\label{6.2simeq}
\Gamma_{G_k,l,6d+1}(X)\otimes_{\Z_l} \Z_l(3d) \simeq \bigoplus_{i:\text{odd}} H^{i}_{\et}(X_{\ol{k}},\Z_l)(\frac{i-1}{2}),
\end{equation}
and if we take $m=6d+2$ then we obtain isomorphism of $\Z_l$-modules with $G_k$ action:
\begin{equation}\label{6.3simeq}
\Gamma_{G_k,l,6d+2}(X)\otimes_{\Z_l} \Z_l(3d+1) \simeq\bigoplus_{i:\text{even}} H^{i}_{\et}(X_{\ol{k}},\Z_l)(\frac{i}{2}).
\end{equation}

There is a $\Z_l$-linear functor $\Gamma_{G_k,l,m}:C_{k,\Z_l}^{\leq d} \to (G_k \curvearrowright \Z_l\text{-mod})$. If $l> \left\{
\begin{array}{ll}
3d+1 & \text{if }\ch(k)=0 \\
2d+e+1 & \text{if }\ch(k)>0
\end{array}
\right.$, then $\Z_l$ is a $\left\{\begin{array}{ll}
\Z[\frac{1}{(3d+1)!}]\text{-algebra} & \text{if }\ch(k)=0 \\
\Z[\frac{1}{(2d+e+1)!}]\text{-algebra} & \text{if }\ch(k)>0
\end{array}
\right.$. By the sequence of functors \eqref{ONIONI} there is a functor 
\begin{equation}\label{redrockch0}
\Gamma_{G_k,l,m}:\KM^{\leq d}(k)_R \to (G_k \curvearrowright \Z_l\text{-mod}) \qquad \qquad \text{   if }\ch(k)=0
\end{equation}
\begin{equation}\label{redrockchp}
\Gamma_{G_k,l,m}:\KM^{\leq d}_{(e)}(k)_R \to (G_k \curvearrowright \Z_l\text{-mod})\qquad \qquad \text{   if }\ch(k)>0
\end{equation}
The following theorem follows from the functor \eqref{redrockch0}, and Tabuada-Marcollini's result (Eq.\eqref{TVd} below) of non-commutative motives of an Azumaya algebra \cite[Theorem~2.1]{motiveazumaya}.

\begin{thm}\label{6.2}
Let $k$ be a field of characteristic $0$. Choose natural numbers $d$ and a prime number $l > 3d+1$, and a algebraic closure $\ol{k}/k$. For smooth projective varieties $X$ and $Y_1,...,Y_j$ whose dimension are less than $d$ and Azumaya algebras $\alpha_i$ on $Y_i$ such that there is a semi-orthogonal decomposition
\[
D^b(X_{\ol{k}}) \overset{\text{SOD}}{\simeq} <D^b(Y_{1,\ol{k}},\alpha_{1,\ol{K}}),...,D^b(Y_{j,\ol{k}},\alpha_{j,\ol{K}})>.
\]
Let us denote $\text{rank}~\alpha_i=r_i$. If $(l,r_i)=1$ for any $i$, then there is a finite extension $M/k$ such that for any finite extension $L/M$ there are isomorphisms of $\Z_l$-module with Galois action of $G_L$:
\[
\bigoplus_{i:\text{even}} H^{i}_{\et}(X_{\ol{k}},\Z_l)(\frac{i}{2}) \simeq  \bigoplus_{e=1}^{j}\bigoplus_{i:\text{even}} H^{i}_{\et}(Y_{e,\ol{k}},\Z_l)(\frac{i}{2})
\]
\[
\bigoplus_{i:\text{odd}} H^{i}_{\et}(X_{\ol{k}},\Z_l)(\frac{i-1}{2}) \simeq  \bigoplus_{e=1}^{j}\bigoplus_{i:\text{odd}} H^{i}_{\et}(Y_{e,\ol{k}},\Z_l)(\frac{i-1}{2}).
\]
\end{thm}
\begin{proof}
By Theorem~\ref{5.2}, there is a finite extension $M/K$ such that there is a semi-orthogonal decomposition 
\[
D^b(X_M)\overset{\text{SOD}}{\simeq}  <D^b(Y_{1,M},\alpha_{1,M}),...,D^b(Y_{j,M},\alpha_{j,M})>,
\] 
and for any finite extension $L/M$, we have a semi-orthogonal decomposition 
\[
D^b(X_L)\overset{\text{SOD}}{\simeq}  <D^b(Y_{1,L},\alpha_{1,L}),...,D^b(Y_{j,L},\alpha_{j,L})>.
\] 
Thanks to \cite[Theorem~2.1]{motiveazumaya}, we know an isomorphism
\begin{equation} \label{TVd}    
U(X_L)_{\Z_l} \simeq \bigoplus_{i=1}^{j}U(Y_{i,L})_{\Z_l}
\end{equation}
in $\KMM(L,{\Z_l})$. In fact, this follows from the above semi-orthogonal decomposition and the fact that $\Z_l$ is $\Z[\frac{1}{ r_1\cdot...\cdot r_j}]$-algebra. By \cite[Corollary 3.1]{M1}, we obtain an isomorphism $(X_L)_{\Z_l} \simeq \bigoplus_{i=1}^{j}(Y_{i,L})_{\Z_l}$ in $\KM^{\leq d}(k,\Z_l)$. The claim follows from the functors \eqref{redrockch0} and isomorphisms \eqref{6.2simeq} and \eqref{6.3simeq}.
\end{proof}

\subsection{} Let us recall the Serre's density conjecture for ordinary reduction.
\begin{conj}(Serre conjecture for ordinary density)\label{Serre2}
Let $X/K$ be a  smooth projective variety over a  number field $K$. Then there is finite extension $L/K$ such that a positive density of primes $v$ of $L$ for which $X_L$ has a good ordinary reduction at $v$.
\end{conj}

We prove the following theorem.


\begin{thm}\label{6.4}
Let $X$ be a cubic $4$-fold which contains $\PP^2$ over a number field $K$. Then $X$ satisfies the conjecture~\ref{Serre2}. Moreover the density of the set of finite primes in which $X_L$ has a good ordinary is one.
\end{thm}
\begin{proof}
In \cite{Kuznetsov}, Kuznetsov proved that there is a K3 surface $S$ over $\ol{K}$ and an Azumaya algebra $\alpha$ of rank $2$ on $S$ satisfying the following semi-orthogonal decomposition (see \cite[Theorem~4.3]{Kuznetsov})
\[
D^b(X_{\ol{K}})\overset{\text{SOD}}{\simeq} <D^b(S,\alpha),\sO_{X_{\ol{K}}},\sO_{X_{\ol{K}}}(1),\sO_{X_{\ol{K}}}(2)>.
\]
Take a finite extension $M/K$ such that $S$ and $\alpha$ are defined by a K3 surface ${S}_M$ and $\alpha_{M}$ over $M$. By Theorem~\ref{5.2}, there is a finite extension $L'/M$ such that there is a semi-orthogonal decomposition 
\[
D^b(X_{L'})\overset{\text{SOD}}{\simeq} <D^b(S_{L'},\alpha_{L'}),\sO_{X_{L'}},\sO_{X_{L'}}(1),\sO_{X_{L'}}(2)>.
\]
Since a K3 surface ${S}_{L'}$ satisfies the conjecture~\ref{Serre2} (see \cite{K3ordinary} or \cite{Joshi}), there is a finite extension $L/L'$ such that a positive density of primes $v$ of $L$ for which $\tilde{S}_L$ has a good ordinary reduction at $v$. By Lemma~\ref{5.15}, there is a semi-orthogonal decomposition 
\[
D^b(X_{L})\overset{\text{SOD}}{\simeq} <D^b(S_{L},\alpha_{L}),\sO_{X_{L}},\sO_{X_{L}}(1),\sO_{X_{L}}(2)>.
\]
For a finite prime $\mathfrak{p}$ on $L$, we consider the following conditions where we write $p$ for the characteristic of the residue field of $\mathfrak{p}$.
\begin{itemize}
    \item[(c-1)] $p >13$.
    \item[(c-2)] $X_L$ and $\tilde{S}_L$ have good reduction at $\mathfrak{p}$.
    \item[(c-3)] The Hodge cohomology $H^i(\mathfrak{X}_{{\mathfrak{p}}},\Omega^j_{\mathfrak{X}_{{\mathfrak{p}}}/\sO_{L_\mathfrak{p}}})$ are free $\sO_{L_{\mathfrak{p}}}$-modules for any $i,j$ where $\mathfrak{X}_{\mathfrak{p}}$ is an integral model of $X_{L_{\mathfrak{p}}}$ over $\sO_{L_{\mathfrak{p}}}$ and $L_{\mathfrak{p}}$ is the completion of $L$ at $\mathfrak{p}$, and $\sO_{L_{\mathfrak{p}}}$ is the integral ring of it.
    \item[(c-4)] $\tilde{S}_L$ has ordinary reduction at $\mathfrak{p}$
    \item[(c-5)] $L_{\mathfrak{p}}$ is unramified over $\Q_p$.
\end{itemize}
We first prove that all but finitely many finite primes $\mathfrak{p}$ satisfy the condition (c-3). Take a non-empty open subscheme $\Spec R \hookrightarrow \Spec \sO_L$ such that $X_L$ has a projective integral model $\mathfrak{X}_R$ over $R$. Consider the subset $U$ of $\Spec R$ of the union of supports of Hodge cohomology group $H^i(\mathfrak{X}_R,\Omega^j_{\mathfrak{X}_R/R})$ as f.g.$R$-modules:
\[
U=\displaystyle\bigcup_{i,j} \text{Supp}~\bigl( \text{Tor}~ H^i(\mathfrak{X}_R,\Omega^j_{\mathfrak{X}_R/R}) \bigr).
\]
The sub set $U$ is finite subset in $\Spec R$ and does`t contain the generic point of $\Spec R$. For a finite prime $\mathfrak{p}$ on $L$ such that $\mathfrak{p}\in \Spec R$, consider the base change diagram of $\mathfrak{X}_R \to \Spec R$ along the natural morphism $R \to \sO_{L_{\mathfrak{p}}}$. Since the morphism of rings $R \to \sO_{L_{\mathfrak{p}}}$ is flat, by the base change theorem of coherent cohomology we have an isomorphism of $\sO_{L_\mathfrak{p}}$-module:
\[
H^i(\mathfrak{X}_R,\Omega^j_{\mathfrak{X}_R/R}) \otimes_R \sO_{L_\mathfrak{p}} \simeq H^i(\mathfrak{X}_{{\mathfrak{p}}},\Omega^j_{\mathfrak{X}_{{\mathfrak{p}}}/\sO_{L_\mathfrak{p}}})
\]
for any $i,j$. Thus if $\mathfrak{p}$ is not in $U$ then $H^i(\mathfrak{X}_{{\mathfrak{p}}},\Omega^j_{\mathfrak{X}_{{\mathfrak{p}}}/\sO_{L_\mathfrak{p}}})$ are f.g.free $\sO_{L_\mathfrak{p}}$-module for any $i,j$. We know the subset $\Spec R \backslash  U$ is an open subscheme of $\Spec \sO_L$, so we obtain that for all but finitely many finite primes,  $\mathfrak{p}$ satisfies the condition (c-3). It is easy to check that for all but finitely many finite primes,  $\mathfrak{p}$ satisfies conditions (c-1), (c-2) and (c-5).

By the definition of the finite extension $L/M$, there is a positive density of primes $v$ of $L$ for which $\tilde{S}_L$ has a good ordinary reduction at $v$. we obtain that the set of finite primes
\begin{center}
$S=$\{$\mathfrak{p}$:finite prime on $L$~$|$~$\mathfrak{p}$ satisfies (c-1) $\sim $ (c-5)\}
\end{center}
has density one. 

Let us prove that for any finite prime $\mathfrak{p}$ in $S$, $X_L$ has ordinary reduction at $\mathfrak{p}$. By the condition (c-1) $p>13=3\times 4+1$, the functors \eqref{redrockch0} and isomorphisms \eqref{6.2simeq} and \eqref{6.3simeq} induces following isomorphisms of $\Z_p$-module with $G_{L_{\mathfrak{p}}}$ action:
\[
\bigoplus_{i:\text{even}} H^{i}_{\acute{e}t}(X_{\widehat{\ol{L}}_{\mathfrak{p}}},\Z_p)(i/2)\simeq \bigoplus_{i:\text{even}}H^{i}_{\acute{e}t}(\tilde{S}_{\widehat{\ol{L}}_{\mathfrak{p}}},\Z_p)(i/2) \oplus \Z_p^{\oplus 2}
\]
\[
\bigoplus_{i:\text{odd}} H^{i}_{\acute{e}t}(X_{\widehat{\ol{L}}_{\mathfrak{p}}},\Z_p)((i+1)/2) \simeq \bigoplus_{i:\text{odd}}H^{i}_{\acute{e}t}(\tilde{S}_{\widehat{\ol{L}}_{\mathfrak{p}}},\Z_p)((i+1)/2)
\]
Choose an isomorphism of fields $\sigma:\widehat{\ol{L}}_{\mathfrak{p}}\simeq \mathbb{C} $, then we have a non-canonical isomorphism of $\Z_p$-modules
\[
H^{i}_{\acute{e}t}(\tilde{S}_{\widehat{\ol{L}}_{\mathfrak{p}}},\Z_p) \simeq H^i_{Sing}(\tilde{S}_{\sigma}^{\text{an}},\Z)\otimes_Z\Z_p
\]
where $\tilde{S}_{\sigma}$ is the base change of $\tilde{S}_{\widehat{\ol{L}}_{\mathfrak{p}}}$ along $\sigma$. Since $\tilde{S}_{\sigma}$ is a complex K3 surface, the singular cohomology $H^i_{Sing}(\tilde{S}_{\sigma}^{\text{an}},\Z)$ is torsion free for $i=0,2,4$, and is $0$ for $i=1,3$. Thus we now know the $\Z_p$-module $H^{i}_{\acute{e}t}(X_{\widehat{\ol{L}}_{\mathfrak{p}}},\Z_p)$ is free $\Z_p$-module for any $i$. By conditions (c-1) and (c-5), we can apply Fontaine-Messing theorem for $X$, we obtain the following isomorphism (see isomorphism~\eqref{fontaine-Messing}):
\[
H^i_{\text{crys}}(X_{L,\kappa(\mathfrak{p})}/W(\kappa(\mathfrak{p})))\simeq H^{i}_{\acute{e}t}(X_{\widehat{\ol{L}}_{\mathfrak{p}}},\Z_p)\otimes_{\Z_p} W(\kappa(\mathfrak{p}))
\]
for any $i$ (for $0\leq i\leq 2\dim X_{L,\kappa(\mathfrak{p})}=8 < 12 \overset{(c-1)}{<} p-1 $, the isomorphism is induced by Fontaine-Messing theorem, and for $i>8$ the both of side are zero). Thus we have that $H^i_{\text{crys}}(X_{L,\kappa(\mathfrak{p})}/W(\kappa(\mathfrak{p})))$ are free $W(\kappa(\mathfrak{p}))$-modules. Thus we can apply Achinger`s result \cite[Proposition 6.7]{AchingerArxiv} for $\mathfrak{X}_{\mathfrak{p}}$ where we use the condition (c-3), i.e. to show $X_{\kappa(\mathfrak{p})}$ is ordinary it is enough to show that the $p$-adic Galois representation $G_{L_{\mathfrak{p}}}\curvearrowright H^{i}_{\acute{e}t}(X_{\widehat{\ol{L}}_{\mathfrak{p}}},\Q_p)$ is ordinary representation for any $i$. By the condition (c-5), we know the $p$-adic Galois representation $G_{L_{\mathfrak{p}}}\curvearrowright H^{i}_{\acute{e}t}(\tilde{S}_{\widehat{\ol{L}}_{\mathfrak{p}}},\Q_p)$ is ordinary representation for any $i$. Since an ordinary $p$-adic representation twisted by the Tate twisted $\Q_p(i)$ is also ordinary, and a direct summand of an ordinary representation is also ordinary, by the isomorphism
\[
\bigoplus_{i:\text{even}} H^{i}_{\acute{e}t}(X_{\widehat{\ol{L}}_{\mathfrak{p}}},\Z_p)(i/2)\simeq \bigoplus_{i:\text{even}}H^{i}_{\acute{e}t}(\tilde{S}_{\widehat{\ol{L}}_{\mathfrak{p}}},\Z_p)(i/2) \oplus \Z_p^{\oplus 2}
\]
we know $H^{i}_{\acute{e}t}(X_{\widehat{\ol{L}}_{\mathfrak{p}}},\Z_p)$ is an ordinary representation if $i$ is even, and we have $\bigoplus_{i:\text{odd}} H^{i}_{\acute{e}t}(X_{\widehat{\ol{L}}_{\mathfrak{p}}},\Z_p)(i/2)=0$.
\end{proof}

\bibliography{bib}
\bibliographystyle{alpha}

%

%
\end{document}